\documentclass[11pt,onecolumn]{article}
\usepackage{amscd}
\usepackage{times}
\usepackage{amsmath}
\usepackage{amssymb}
\usepackage{xspace}
\usepackage{theorem}
\usepackage{graphicx}
\usepackage{ifpdf}
\usepackage{url,hyperref}
\usepackage{latexsym}
\usepackage{euscript}
\usepackage{xspace}
\usepackage[all]{xy}
\usepackage{color}
\usepackage{makeidx}
\usepackage{tikz}
\long\def\remove#1{}
\newtheorem{theorem}{Theorem}[section] 
\newtheorem{lemma}[theorem]{Lemma}

\newtheorem{obs}[theorem]{Observation}

\newtheorem{definition}[theorem]{Definition}
\newtheorem{proposition}[theorem]{Proposition}
\newenvironment{proof}{{\em Proof:}}{\hfill{\hfill\rule{2mm}{2mm}}}

\newcommand {\mm}[1] {\ifmmode{#1}\else{\mbox{\(#1\)}}\fi}

\newcommand{\A}                        {\mathrm {\mathbb{A}}}

\newcommand{\img}{\mathrm img}
\newcommand{\supp}{\mathrm supp}
\newcommand{\coker} {\mathrm coker}

\newcommand{\cancel}[1]

\begin{document}

\title{A refinement of Betti numbers  and homology in the  presence of a continuous function. I}

\author{
Dan Burghelea  \thanks{
Department of Mathematics,
The Ohio State University, Columbus, OH 43210,USA.
Email: {\tt burghele@math.ohio-state.edu}}
}
\date{}

\date{}
\maketitle
\begin{abstract}
\vskip .2in
We propose a refinement of the Betti numbers and of the homology with coefficients in a field of a compact ANR  $X,$ in the presence of a continuous real valued function on $X.$
The refinement of Betti numbers consists of  finite configurations of points  with multiplicities in the complex plane whose total cardinality are the Betti numbers and the refinement of homology consists of  configurations of vector spaces indexed by points in complex plane, with the same support as the first, whose direct sum is isomorphic to the homology. When the homology is  equipped with a scalar product these vector spaces are canonically realized as mutually orthogonal subspaces of the homology.

The assignments above are in analogy with the collections of eigenvalues and generalized eigenspaces of a linear  map in a finite dimensional complex vector space.

A number of remarkable properties of the above configurations are  discussed. 

\end{abstract}
\thispagestyle{empty}
\setcounter{page}{1}

%
\tableofcontents

\section {Introduction } \label {I}

The results of this paper and its subsequent part II, mostly  obtained  in collaboration with Stefan Haller, provide
a shorter version of some results  of paper \cite{BH}, still unpublished,  extend their generality based  on the involvement of the topology of Hilbert cube manifolds and refine them as configurations of complex numbers and of vector spaces. 

Precisely, for a fixed field $\kappa$ and $r\geq 0,$ one proposes a refinement of the Betti numbers $b_r(X)$ of a compact ANR  $X$ \footnote {see the definition of an ANR in subsection \ref{SS22}}  and a refinement of the homology $H_r(X)$ with coefficients in the field $\kappa$  
in the presence of a continuous function  $f: X\to \mathbb R.$ 

The refinements consists of  finite configurations of points  with multiplicity located in the plane $\mathbb R^2= \mathbb C,$  denoted by $\delta^f_r,$ 
equivalently of monic polynomials with complex coefficients $P^f_r(z),$  of degree  the Betti numbers $b_r(X),$   and  finite configurations of $\kappa-$vector spaces denoted by $\hat \delta^f_r$ with the same support and direct sum of all vector spaces 
isomorphic to $H_r(X),$ cf. Theorem \ref{T1}. 
The points of the configurations $\delta^f_r,$ equivalently  the zeros of the polynomials $P^f_r(z),$ are complex numbers $z= a+ib\in \mathbb C$ with both $a,b$  critical values \footnote {see section 2.2. below for the definition of regular and critical value}, cf. Theorem \ref{T1}.  The two configurations are related by $\dim \hat\delta^f_r= \delta^f_r.$

We show that :
\begin{enumerate}
\item
The assignment $f\rightsquigarrow P_r^f(z)$ is  continuous when $f$ varies in the space of continuous maps equipped with the compact open topology, cf. Theorem \ref{T2}. 
\item For an open and dense subset of continuous maps (defined on $X,$ an ANR satisfying some mild properties,) the  points of the configurations $\delta^f_r$ or the zeros of the polynomials $P^f_r(z)$ have multiplicity one, cf. Theorem \ref{T1}.
\item When $X$ is a closed topological $n-$manifold the Poincar\'e Duality between the Betti numbers $\beta_r$ and $\beta_{n-r}$ gets refined to a Poincar\'e Duality between configurations $\delta^f_r$and $\delta^f_{n-r}$ and the Poincar\'e Duality between $H_r(X)$ and $H_{n-r}(X)^\ast$  to a Poincar\'e Duality between configurations $\hat \delta^f_r$and $(\hat \delta^f_{n-r})^\ast $, cf .Theorem \ref{T3}.   
\item 
For  each  point of the configuration $\delta^f_r,$ equivalently  zero $z$ of the polynomial $P^f_r(z),$ 
the assigned vector space $ \hat\delta^f_r(z)$ has dimension the multiplicity of $z$ and  is a quotient of vector subspaces $ \hat\delta^f_r(z)= \mathbb F_r(z) /\mathbb F'_r(z),$  $\mathbb F'_r(z)\subset \mathbb F_r(z)\subset H_r(X).$ When  $\kappa=\mathbb R$ or $\mathbb C$ and $H_r(X)$ is equipped with a Hilbert space structure $ \hat\delta^f_r(z)$ identifies canonically to a subspace ${\bf H}_r(z)$ of $H_r(X)$ s.t. ${\bf H}_r(z)\perp {\bf H}_r(z')$ for $z\ne z'$ and $\oplus_z {\bf H}_r(z)= H_r(X),$ cf. Theorem \ref{T1}. This provides an additional structure (direct sum decomposition of $H_r(X)$ (which in view of Theorem \ref{T1}, for a generic $f,$  has all components of dimension 1).
\end{enumerate}

We refer to the system $(H_r(X), P_r^f(z),  \hat \delta^f_r)$ as the {\it $r-$homology spectral package of $(X,f)$}  in analogy with the spectral package of $(V,T),$ $V$ a vector space 
$T$ a linear endomorphism, which consists of the characteristic polynomial $P^T(z)$ with its roots $z_i,$   the eigenvalues of $T$ and with their corresponding generalized eigenspaces $V_{z_i}.$

In case $X$ is the underlying space of a closed oriented Riemannian manifold $(M^n, g)$  and $\kappa= \mathbb R$ or $\mathbb C$ the vector space $H_r(M^n),$ via the identification with the harmonic $r-$forms, has a structure of Hilbert space. The configuration $\hat\delta^f_r,$ for $f$ generic, provides a base in the space of harmonic forms.

All these results  are collected in the main theorems below, Theorems {\ref{T1}, \ref{T2}, \ref{T3}.  
 Theorems \ref{T1} (1) and (3), Theorem \ref{T2} and Theorem \ref{T3} were   established in \cite {BH}, not yet in print,  but under more restrictive hypothesis like "$X$ homeomorphic to a simplicial complex'' or "$f$ a tame map". 
In this paper we removed this hypothesis using results on Hilbert cube manifolds reviewed in subsection \ref {SS23} and complete them 
with additional results.

It is worth to note that the points of the configurations $\delta^f_r$  located above and on the diagonal in the plane $\mathbb R^2$ determine and are determined by the closed $r-$bar codes in the level persistence of $f$ while those below diagonal are determined and determine the open $(r-1)-$bar codes in the level persistence as observed in \cite {BH}. The algorithms proposed in \cite {CSD09} and in \cite {BD11} can be used for their calculation. 

Similar refinements hold for angle valued maps and will be discussed in Part II. In this case the homology has to be replaced by either the Novikov homology of $(X, \xi_f)$ which in our work is a f.g free module over the ring of Laurent polynomials $\kappa[ t^{-1},t]$ 
or, in case $\kappa$ is $\mathbb R$ or $\mathbb C,$ by the $L_2-$homology of the infinite cyclic cover defined by $\xi_f\in H^1(X:\mathbb Z),$ determined  by $f.$ In this case the $L_2-$homology is regarded as a Hilbert module over the von-Neumann algebra associated to the group $\mathbb Z.$  In this case ${\bf H}_r(z)$ are Hilbert submodules, and $\delta^f_r(x)$ is the von Neumann dimension of ${\bf H}_r(z).$ Note that the $L_2-$Betti numbers are actually the Novikov--Betti numbers of $(X,\xi_f)$ (which agree with the rank of the corresponding free module). 

The Author thanks  S. Ferry for help  in clarifying  a number of aspects about Hilbert cube manifolds and ANR's.
The Author is equally grateful to the referee for many suggestions, requests for clarifications and sometimes alternative arguments.

\section{Preliminary definitions}

\subsection {Configurations}\label {C}

Let $X$ be a topological space.

A {\it finite configuration of points} in $X$ is a map $$\delta:X\to \mathbb Z_{\geq 0}$$ with finite support.

A {\it finite configuration of vector spaces  indexed by points in $X$ } is a map  with finite support $$\overline \delta: X\to \rm{VECT}$$
(i.e. $\hat \delta(x)= 0 $ for all but finitely many $x\in X$),  where  
VECT denotes the collection of $\kappa-$vector spaces. 

For $N$ a positive integer number  denote by $\mathcal C_N(X)$ the set of configurations of points in $X$ with total cardinality $N,$
$$\mathcal C_N(X):=\{\delta: X\to \mathbb Z_{\geq 0} \mid \sum_{x\in X} \delta (x)= N\}.$$
 
For $V$ a finite dimensional $\kappa-$vector space denote by 
  $\mathcal P(V)$ be the set of  subspaces of $V$ 
and by ${\bf\mathcal C}_V(X)$ the set
$${\bf\mathcal C}_V(X):=\{ \overline\delta: X\to \mathcal P(V) \mid \begin{cases} \sharp \ \{x\in X \mid  \overline \delta (x)\ne 0\} <\infty\\
\overline\delta(x)\cap \sum_{y\ne x} \overline \delta(y)=0\\
\sum_{x\in X} \overline \delta(x)= V\end{cases}\}.$$

Here $\sharp$ denotes cardinality of the set in parentheses.

One  considers  the map 
$$e:\bf{\mathcal C}_V(X) \to \mathcal C_{\dim V}(X)$$  defined by $$e(\overline\delta) (x)= \dim \overline \delta(x)$$ and call  the configuration $e(\overline \delta)$ the {\it dimension} of $\overline \delta.$ 

Both sets  sets $\mathcal C_N(X)$ and ${\bf \mathcal C}_V (X)$   can be equipped with natural topology  ({\it collision topology}).
One  way to describe these topologies is to specify for each $\delta$ or $\hat\delta$  a system of {\it fundamental neighborhoods}.
If $\delta$ has as support  the set of points $\{x_1, x_2, \cdots,  x_k\},$ a fundamental neighborhood $\mathcal U$ of $\delta$ is specified by a collection of $k$ disjoint open neighborhoods  $U_1, U_2,\cdots, U_k$ of $x_1,\cdots,  x_k$ and consists of $\{\delta'\in \mathcal C_N(X)\mid
\sum_{x\in U_i} \delta'(x)=\delta(x_i)\}.$ 
Similarly if $\overline \delta$ has as support  the set of points $\{x_1, x_2, \cdots,  x_k\}$ with $\overline \delta(x_i)= V_i\subseteq V\},$ a fundamental neighborhood $\mathcal U$  of $\overline \delta$ is specified by a collection of $k$ disjoint open neighborhoods  $U_1, U_2,\cdots, U_k$ of $x_1,\cdots,  x_k,$ and consists of 
$$\{\overline \delta'\in \mathcal C_V(X) \mid  x\in U_i\Rightarrow \overline \delta'(x)\subset V_i,  
\bigoplus_ {x\in U_i} \overline \delta' (x) =V_i\}.$$ Clearly $e$ is continuous. 
 
 When $\kappa$ is an infinite field the topology of $\mathcal C_V(X)$ has too many connected components to be useful unless the geometry forces 
 the possible values of the configurations to be at most countable.
 
When $\kappa= \mathbb R$ or $\mathbb C$ and $V$ is a   Hilbert space it is natural to consider the subset of $\mathcal C^O_V(X)\subset \mathcal C_V(X)$ consisting of configurations whose vector spaces  $\overline \delta (x)$ are mutually orthogonal. In this case  for $\overline \delta$  with support the set of points $\{x_1, x_2, \cdots,  x_k\}$ and  $\overline\delta(x_i)= V_i\subseteq V,$ one can consider a fundamental neighborhood $\mathcal U$  of $\overline \delta$ is specified by a collection of $k$ disjoint open neighborhoods  $U_1, U_2,\cdots U_k$ of $x_1,\cdots,  x_k,$ and open neighborhoods $O_1, O_2, \cdots, O_k$ of $V_i$ in $G_{\dim V_i} (V)$ and  consists of 
$$\{\overline \delta'\in \mathcal C^O_V(X) \mid  \bigoplus_ {x\in U_i} \hat \delta' (x) \in O_i\}.$$
Here $G_k(V)$ denotes the Grassmanian of $k-$dimensional subspaces  of $V.$

With respect to this topology  $e$ is continuous, surjective and proper, with fiber above $\delta,$ the subset of  $G_{n_1}(V)\times G_{n_2}(V) \cdots \times G_{n_k}(V)$ consisting  of $(V'_1, V'_2,\cdots, V'_k), V'_i\in G_{n_i}(V)$ mutually orthogonal,   where $n_i= \dim V_i.$  This set is compact and is actually an algebraic variety. 

Note that: 
\begin{enumerate}
\item $\mathcal C_N(X)= X^N/\Sigma_N$ is the so called $N-$symmetric  product and if $X$ is  a metric space with distance $D$ 
then the collision topology is  the topology defined by the distance $\underline D$ on  $X^N/\Sigma_N$ 
induced from the distance on $X^N$ given by 
$D(x_1,x_2, \cdots, x_N; y_1, y_2, \cdots y_N):= sup _{i=1,\cdots, N} \{D(x_i, y_i)\}.$
\item If $X= \mathbb R^2= \mathbb C$ then $\mathcal C_N(X)$ identifies to the set of monic polynomials with complex coefficients. To the configuration $\delta$ whose support consists of the points $z_1, z_2, \cdots z_k$ with  $\delta(z_i)= n_i$ one associates the monic polynomial  $P^f(z)= \prod _i (z- z_i)^{n_i}. $ Then $\mathcal C_N(X)$ identifies to $\mathbb C^N$ as metric spaces. 
 \item  The space ${\bf \mathcal C}_V (X)$ and then ${\bf \mathcal C}_V (\mathbb R^2)$ can be  equipped with a complete metric which induces the collision topology but this will not be used here. 
\end{enumerate}

\subsection {Tame maps}\label {SS22}

Recall that a metrizable space $X$ is an ANR if any closed subset $A$ of a metrizable space $B$ with $A$ homeomorphic to $X$ has a neighborhood $U$ which retracts to $A$, cf \cite {Hu} chapter 3. 
Recall also that any space homeomorphic to a locally finite simplicial complex  or a finite dimensional topological manifold or an infinite dimensional manifold (i.e. a paracompact Hausdorff space locally homeomorphic to the Hilbert space $l_2$ or the Hilbert cube $I^\infty$) is an ANR, cf \cite {Hu}.

{\bf All maps $f:X\to \mathbb R$ in this paper are continuous proper maps defined on $X$ an ANR,}  hence if such maps exists $X$ is locally compact.
  From now on the words ''proper continuous'' should always be assumed to precede the word "map" even if not specified.  

The following concepts are consistent with the familiar terminology in topology.

\begin{enumerate}
\item A map $f:X\to \mathbb R$ is {\it weakly tame}  if for any $t\in \mathbb R,$  the level $f^{-1}(t)$ is an ANR.
 Therefore for any bounded or unbounded closed interval $I=[a,b], \ a,b \in \mathbb R \sqcup \{\infty, -\infty\}$  $f^{-1}(I)$ is an ANR. 
Indeed if $I=[a,b],$ in view of the hypothesis that  $f^{-1}(a)$ and  $f^{-1}(b)$ are ANRs  and  of the definition of ANR,  there exists an open set $U\subset X\setminus f^{-1}(a,b)$ which retracts to $f^{-1}(a)\sqcup f^{-1}(b).$  Then $U\cup f^{-1}[a,b]$ is an open set in $X$ which retracts to $f^{-1}(I).$ Since $X$ is an ANR this suffices to conclude that  $f^{-1}(I)$is an ANR cf \cite {Hu}.   A similar argument can be used for $I=(-\infty, a]$ or $I= [b, \infty).$
 
\item The number $t\in \mathbb R$ is a {\it regular value} if there exists $\epsilon >0$ s.t. for any $t'\in (t-\epsilon, t+\epsilon)$ the open set $f^{-1}(t-\epsilon, t+\epsilon)$ retracts by deformation to 
$f^{-1}(t').$
A number $t$ which is not regular value is a {\it critical value}. In view of hypothesis on $f$ a "map" (hence $X$ locally compact and $f$ proper)  the requirement on $t$ in the definition of {\it weakly tame}  is satisfied for any $t$ regular value. 
Informally, the critical values are the values $t$ for which the topology of the level (= homotopy type) changes.  One denotes by $Cr(f)$ the collection of critical values of $f.$
 
\item The map $f$ is called {\it tame} if weakly tame and in addition: 

(a) The set of  critical values $Cr(f)\subset \mathbb R$ is discrete,  

(b) $\epsilon (f):= \inf \{|c-c'| \mid c,c'\in Cr(f), c\ne c'\}$ satisfies $\epsilon(f)>0.$  

If $X$is compact then (a)  implies (b).

\item An ANR which has the 
tame maps dense in the set of all maps w.r. to the fine $C_0-$ topology is called a {\it good ANR}.

There exist  compact ANR's (actually compact homological n-manifolds, cf \cite {DW}) with no co-dimension one subsets which are ANR's, hence compact ANR's which are not {\it good }. 

\end{enumerate}

The reader should be aware of the following rather obvious facts.
\begin{obs}\label {O21}\
\begin{enumerate}
\item If $f$ is weakly tame  map then $f^{-1}([a,b])$ is a compact ANR and has the homotopy type of a finite simplicial complex (cf \cite{Mi}) and therefore finite dimensional homology w.r. to any field $\kappa.$
\item If $X$ is a locally finite simplicial complex and $f$ is a simplicial map then $f$ is weakly tame  with the set of critical values discrete. Critical values are among the values of $f$ on vertices. If in addition $X$ is compact then $f$ is tame.
\item If $X$ is homeomorphic to a finite simplicial complex then the set of tame maps 
is dense in the set of all continuous maps with the $C_0-$ topology (= compact open topology). The same remains true if $X$ is a compact Hilbert cube manifold defined in the next section. In particular all these spaces are good ANR's.
\end{enumerate}
\end{obs}

For the needs of this paper  weaker than usual concepts of regular or critical values and tameness, relative to homology with coefficients in the field $\kappa$ suffice.  They are introduced  in section 3.

\subsection {Compact Hilbert cube manifolds} \label{SS23}\

Recall that: 

-  The Hilbert cube $Q$ is the infinite product $Q=I^\infty=  \prod_{i\in \mathbb Z_{\geq 0}} I_i$ with $I_i= [0,1].$ 
The  topology of $Q$ is  given by the distance $d(\overline u ,\overline v)= \sum _i |u_i- v_i|/ 2^i$ with $\overline u= \{u_i\in I, i\in \mathbb Z_{\geq 0}\}$  and  $\overline v= \{v_i\in I, i\in \mathbb Z_{\geq 0}\}$

-  The space $Q$ is a compact ANR and so is any $X\times Q$ for any  $X$  compact ANR.

-  A compact Hilbert cube manifold is a compact Hausdorff space locally homeomorphic to the Hilbert cube $Q.$
\vskip .1in

For $f:X\to \mathbb R$ and $F:X\times Q\to \mathbb R$ denote by $\overline f_Q :X\times Q\to \mathbb R$  and  $F_k:X\times Q\to \mathbb R$
the maps defined by
$$\overline f _Q(x,\overline u)= f(x)$$
and 
$$F_k(x,\overline u)= F(x, u_1, u_2, \cdots u_k, 0,0,\cdots).$$

In view of the definition of $\overline f_Q$ and of the metric on $Q$ observe that : 

\begin{obs}\label {OA}\label{O22}\
\begin {enumerate}
\item If $f:X\to \mathbb R$ is a tame map so is $\overline f_Q.$  
\item  If $X$ is compact then the sequence of maps $F_n$ is  uniformly convergent to  the map $F$ when $n\to \infty.$ 
\end{enumerate}
\end{obs}

The following are basic  results about compact Hilbert cube manifolds whose proof can be found in \cite {CH}.

\begin{theorem} \label {T23}\ 
\begin{enumerate} 
\item (R Edwards) If $X$ is a compact ANR then $X\times Q$ is a compact Hilbert cube manifold.
\item (T.Chapman) Any  compact Hilbert cube manifolds is homeomorphic to $K\times Q$ \ for some finite simplicial complex $K.$ 
\item (T Chapman) If $\omega:X\to Y$ is a homotopy equivalence between two finite simplicial complexes with Whitehead torsion $\tau(\omega)=0$ then  the there exists a homeomorphism $\omega': X\times Q\to Y\times Q$ s.t. 
$\omega'$ and $\omega\times id_Q$ are homotopic.  
As a consequence of Observation \ref{O24} below, two compact Hilbert cube manifolds which are  homotopy equivalent  become homeomorphic after product with $\mathbb S^1.$
\end{enumerate}
\end{theorem}

\begin{obs} \label {O24}
 (folklore)  If $\omega$ is a homotopy equivalence between two finite simplicial  complexes then $\omega\times id_{\mathbb S^1}$ 
has the Whitehead torsion $\tau(\omega\times id_{\mathbb S^1}) =0.$
\end{obs} 

As a consequence of the above statements we have the following  proposition.
\begin{proposition}\label {PA}\
Any  compact Hilbert cube manifold $M$ is a good ANR. 
\end{proposition}

\begin{proof} 
\ A map $f: M\to \mathbb R,$ $M$ a compact Hilbert cube manifold, is called {\it special} if there exist a finite simplicial complex $K,$  a map $g:K\to \mathbb R$ and a homeomorphism $\theta: M\to K\times Q$ s.t. $\overline g\cdot \theta= f$ and a special map is p.l. \footnote {p.l.= piecewise linear}if in addition $g$ is p.l. map.
By Observation \ref{O22} any map $f:M\to \mathbb R$ is $\epsilon/2$ closed to a special map. Since any continuous real valued map defined on a simplicial complex $K$ is $\epsilon/2$ close to a p.l. map then any special map on $M$ is $\epsilon/2$ closed to a special p.l.  map. 
Consequently $f$ is $\epsilon$ closed to a special p.l. map which is tame  in view of Observations \ref{O21} and \ref{O22}. This implies that the set of tame maps is dense in the set of all continuous maps. 
\end{proof}

\section {The configurations $\delta^f_r$ and $\hat \delta ^f_r$}\label {P}

In this paper we fix a field $\kappa,$ and for a space $X$ denote by $H_r(X)$ the homology of $X$ with coefficients in the field $\kappa.$  
Let $f:X\to \mathbb R$ be a  map. As in the previous section $f$ is proper continuous and  $X$ is  a locally compact  ANR. One
 denotes by :
\begin{enumerate}
\item  $X_a$ the sub level 
$X_a: =f^{-1}(-\infty,a]),$ 
\item $X^b$ the super level 
$X^b:=f^{-1}([b,\infty)),$
\item $\mathbb I^f_a(r):= \img (H_r(X_a)\to H_r(X)) \subseteq H_r(X),$
\item $\mathbb I^b_f(r):= \img (H_r(X^b)\to H_r(X))\subseteq H_r(X),$
\item $\mathbb F_r^f(a,b)= \mathbb I^f_a(r)\cap \mathbb I^b_f(r)\subseteq H_r(X).$
\end{enumerate}

Clearly one has the following observation.
\begin{obs}\label {O31}\
\begin{enumerate}
\item For $a'\leq a$ and $b\leq b'$ one has $\mathbb F_r^f(a',b')\subseteq\mathbb F_r^f(a,b),$
\item For $a'\leq a$ and  $b\leq b'$ one has $\mathbb F_r^f(a',b)\cap \mathbb F_r^f(a,b')= \mathbb F_r^f(a',b').$
\item $\sup_{x\in X} |f(x)-g(x)|
<\epsilon$ implies $\mathbb F^g(a-\epsilon, b+\epsilon) \subseteq\mathbb F^f_r(a,b).$
\end{enumerate}
\end{obs}

Note that we also have the following  proposition. 
\begin{proposition}\label {P32}
If $f$ is a map as above 
then $\dim \mathbb F^f_r(a,b) <\infty.$ 
\end{proposition}

\begin{proof} If $X$ is compact there is nothing to prove since $H_r(X)$ has finite dimension.
Suppose $X$ is not compact.  
In view of Observation \ref {O31} item 1. it  suffices to check the statement for $a > b.$ 
If $f$ is weakly tame in view of Observation \ref{O21} $X_a,$ $X^b$ and $X_a\cap X^b$ are ANR's, with $X_a\cap X^b$ compact and $X=X_a\cup X^b,$ hence the Mayer-Vietoris long exact sequence in homology is valid.  Denote by $i_a(r): H_r(X_a)\to H_r(X) $ and $i^b(r): H_r(X^b)\to H_r(X) $ the inclusion induced linear maps and observe that $\mathbb F_r(a,b):=\mathbb I_a\cap \mathbb I^b\subseteq (i_a(r))(\ker(i_a(r)- i^b(r)).$ In view of Mayer-Vietoris sequence in homology $\ker(i_a(r)- i^b(r))$ is isomorphic to a quotient of the vector space of $H_r(X_a\cap X^b),$ hence of finite dimension, and  the result holds. 

If $f$ is not weakly tame one argue as follows.
It is  known that  any  $X$ a locally compact ANR  is proper homotopy dominated with respect to any open cover by some locally finite simplicial complex $K.$  cf \cite{Ba75}
\footnote{as a replacement for an argument based on an incorrect reference the above argument and the reference was proposed by the referee}. 
Choose such cover for example  $f^{-1}(n-1, n+1)_{n\in \mathbb Z}$ and such a homotopy domination $\xymatrix{ X\ar[r]^i &K\ar[r]^\pi &X}$ for this cover. Choose $g:K\to \mathbb R$ 
 a proper simplicial approximation of $f\cdot \pi$ (hence tame), and $a'>a$ and $b'<b$ such that  $i (X^f_a) \subset K^g_{a'},$ $i (X_f^b) \subset K_g^{b'}.$ Then $\mathbb F^f_r(a,b)$ is isomorphic to a subspace of $\mathbb F^g_r(a', b').$  
Since  the dimension of $\mathbb F^g_r(a', b')$  is finite so is the dimension of $\mathbb F^f_r(a,b).$
\end{proof}

\begin{definition}\label {D1}
A real number $t$  is  a {\bf  homologically regular value} if there exists $\epsilon (t) >0$  s.t. for any  $0<\epsilon ,\epsilon(t) $
the inclusions $\mathbb I^f_{t-\epsilon}(r)\subset\mathbb I^f_t(r)\subset  \mathbb I^f_{t+\epsilon}(r)$  and 
$\mathbb I_f^{t-\epsilon}(r)\supset\mathbb I_f^t(r)\supset  \mathbb I_f^{t+\epsilon}(r)$ are equalities and a {\bf homologically critical value} if not a homologically regular value.
\end{definition}
Denote by $CR(f)$ the set of all homologically critical values. If $f$ is weakly tame  then $CR(f)\subseteq Cr(f).$ 

\begin{proposition} \label{P334}
If $f:X\to \mathbb R$ is a map (hence $X$ is ANR and $f$ is proper)  then $CR(f)$ is discrete. 
\end{proposition} 
\begin{proof} %
As pointed out above  in the proof of Proposition \ref{P32} one can find a proper simplicial map $g:K\to \mathbb R$ and a proper homotopy domidominationnation $\alpha: K\to X$ s. t. $|f\cdot \alpha - g| <M.$ If so for any $a<b,$ $a,b \in \mathbb R$  one has $\dim (\mathbb I^f_b(r)/ \mathbb I^f_a(r))\leq \dim ( \mathbb I^g_{b+M}(r)/ \mathbb I^g_{a-M}(r)) \leq \dim (H_r(g^{-1}([a-M, b+M]), g^{-1}(a-M)) <\infty,$ which implies that there are only  finitely many changes in $\mathbb I^f_t(r)$ for $t$ with $a\leq t\leq b,$ 
Similar arguments show that there are only finitely many changes of $\mathbb I^t_f(r)$ for $t$ with $a\leq t\leq b.$ This suffices to have $CR(f)\cap [a,b]$ a finite set for any $a<b$ hence $ CR(f)$ discrete.
 \end{proof}
\begin{definition}
Define  by $\tilde \epsilon (f): = \inf |c'- c''|, c',c''\in CR(f), c'\ne c''$ and 
call $f$ homologically tame (w.r. to $\kappa$ if $\tilde \epsilon (f)>0.$ 
\end{definition} 
Clearly tame maps are  homologically tame w.r. to any field $\kappa$ and $\tilde \epsilon(f) >\epsilon(f).$

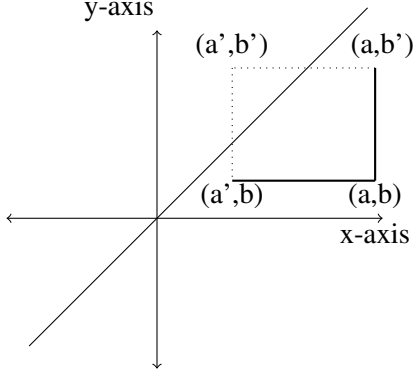
\begin{figure}[t]
\begin{tikzpicture}
\draw [<->]  (0,2.5) -- (0,0) -- (3,0);
\node at (-0.5,2.8) {y-axis};
\node at (2.9,-0.2) {x-axis};
\node at (1,0.3) {(a',b)};
\node at (2.9,0.3) {(a,b)};
\node at (3,2.3) {(a,b')};
\node at (1,2.3) {(a',b')};
\draw [<->]  (0,-2) -- (0,0) -- (-2,0);
\draw [thick] (1,0.5) -- (2.9,0.5);
\draw [thick] (2.9,0.5) -- (2.9,2);
\draw [dotted] (2.9,2) -- (1,2);
\draw [dotted] (1,2) -- (1,0.5);
\draw (0,0) -- (2.8,2.8);
\draw (0,0) -- (-1.7,-1.7);
\end{tikzpicture}
\caption {The {\it box} $ B : =(a',a]\times [b,b')\subset \mathbb R^2$  }
\end{figure}
\noindent Consider the  sets of the form $B= (a',a]\times [b,b')$  with $a'<a, b<b'$ and refer to $B$  as {\it  box} (Figure 1). 

\noindent To a box $B$ as above we assign the quotient  of subspaces $$\mathbb F^f_r(B):= \mathbb F^f_r(a,b)/ \mathbb F^f_r(a',b)+\mathbb F^f_r(a,b'),$$
and denote by 
\begin{equation*}
F^f_r(a,b): = \dim \mathbb F^f_r(a,b), \quad  F^f_r(B): = \dim \mathbb F^f_r(B).
\end{equation*}
In view of Observation \ref{O31} item 2. one has 
 $$F^f_r(B):=  F^f_r(a,b) + F^f(a',b') - F^f_r(a',b)-  F^f(a,b') .$$

It will also be convenient  to denote by 
$$ (\mathbb F^f_r)'(B):= \mathbb F^f_r(a',b)+ \mathbb F^f_r(a,b')\subseteq \mathbb F^f_r(a,b),$$ in which case $$\mathbb F^f_r(B)= \mathbb F^f_r(a,b)/(\mathbb F^f_r)'(B).$$ We denote by $\pi^B_{ab,r}$ 

\begin{equation}
\pi^B_{ab,r}: \mathbb F^f_r(a,b)\to \mathbb F_r^f(B)
\end{equation} the obvious projection.

To ease the writing, when  no risk of ambiguity, one  drops  $f$ 
from notations. 

\vskip .1in
If  $\kappa= \mathbb R$ or $\mathbb C$ and  $H_r(X)$  is equipped with inner product (non degenerate positive definite  hermitian scalar product) one 
denotes by  ${\bf H}_r(B)$
the orthogonal complement of $\mathbb F'_r(B)= (\mathbb F_r(a',b)+\mathbb F(a,b'))$ inside  $\mathbb F_r(a,b),$ which is a Hilbert space being finite dimensional,
and one  has $${\bf H}_r(B)\subseteq \mathbb F_r(a,b)\subseteq H_r(X).$$ 

\newcommand{\mynewnewpicture}[1][ ]{
\begin{tikzpicture} [scale=1]
\draw [dashed, ultra thick] (0,0) -- (0,3);
\draw [line width=0.10cm] (5,0) -- (5,3);
\draw [line width=0.10cm] (2,1) -- (2,3);
\draw [line width=0.10cm] (0,1) -- (2,1);
\draw [line width=0.10cm] (0,0) -- (5,0);
\draw [dashed, ultra thick] (0,3) -- (5,3);
\node at (1,2) {B'};
\node at (3.5,0.8) {B''};
\node at (2.5, -0.5) {Figure 4};
\end{tikzpicture}
\hskip .5in
\begin{tikzpicture} [scale=1]
\draw [dashed, ultra thick] (0,0) -- (0,3);
\draw [line width=0.10cm] (5,0) -- (5,3);
\draw [dashed, ultra thick] (3,1.5) -- (5,1.5);
\draw [dashed, ultra thick] (3,0) -- (3,1.5);
\draw [line width=0.10cm] (0,0) -- (5,0);
\draw [dashed, ultra thick] (0,3) -- (5,3);
\node at (1,2) {B''};
\node at (3.5, 0.8) {B'};
\node at (2.5, -0.5) {Figure 5};
\end{tikzpicture}}

\newcommand{\mynewpicture}[1][ ]{
\begin{tikzpicture} [scale=1]
\draw [dashed, ultra thick] (0,0) -- (0,3);
\draw [line width=0.10cm] (5,0) -- (5,3);
\draw [line width=0.10cm] (2,1) -- (2,3);
\draw [line width=0.10cm] (0,1) -- (2,1);
\draw [line width=0.10cm] (0,0) -- (5,0);
\draw [dashed, ultra thick] (0,3) -- (5,3);
\node at (1,2) {B'};
\node at (3.5,0.8) {B''};
\node at (2.5, -0.5) {Figure 4};
\end{tikzpicture}
\hskip .5in
\begin{tikzpicture} [scale=1]
\draw [dashed, ultra thick] (0,0) -- (0,3);
\draw [line width=0.10cm] (5,0) -- (5,3);
\draw [dashed, ultra thick] (3,1.5) -- (5,1.5);
\draw [dashed, ultra thick] (3,0) -- (3,1.5);
\draw [line width=0.10cm] (0,0) -- (5,0);
\draw [dashed, ultra thick] (0,3) -- (5,3);
\node at (1,2) {B''};
\node at (3.5, 0.8) {B'};
\node at (2.5, -0.5) {Figure 5};
\end{tikzpicture}}

\newcommand{\mypicture}[1][ ]{
\begin{tikzpicture} [scale=1]
\draw [line width=0.10cm] (0,0) -- (5,0);
\draw [line width=0.10cm] (5,0) -- (5,3);
\draw [dashed, ultra thick] (0,3) -- (5,3);
\draw [dashed, ultra thick] (0,0) -- (0,3);
\draw [line width=0.10cm] (3,0) -- (3,3);
\node at (1.5,1.5) {B1};
\node at (4,1.5) {B2};
\node at (2.5, -0.5){Figure 2};
\end{tikzpicture}
\hskip  .5in
\begin{tikzpicture}[ ] [scale=0.8]
\draw [line width=0.10cm] (7,0) -- (12,0);
\draw [line width=0.10cm] (12,0) -- (12,3);
\draw [dashed, ultra thick] (7,0) -- (7,3);
\draw [dashed, ultra thick] (7,3) -- (12,3);
\draw [line width=0.10cm] (7,1) -- (12,1);
\node at (9.5,2) {B1};
\node at (9.5,0.5) {B2};
\node at(9.5, -0.5) {Figure 3};
\end{tikzpicture}}

\begin{proposition}\label {P1}
Let $a''<a'<a ,$  $b< b'$ and $B_1,$ $B_2$ and $B$  the boxes $B_1= (a'',a']\times [b,b''),$ $B_2= (a',a]\times [b,b')$
and $B= (a'',a]\times [b,b') $ (see Figure 2).

1. The inclusions $B_1\subset B$ and $B_2\subset B$ induce the linear maps 
\begin{equation}
i^B_{B_1,r}: \mathbb F_r(B_1)\to  \mathbb F_r(B)\end{equation} and \begin{equation} \pi^{B_2}_{B,r}:  \mathbb F_r(B)\to  \mathbb F_r(B_2)\end{equation}  
such that the following sequence is exact
$$\xymatrix{0\ar [r]& \mathbb F_r(B_1)\ar[r]^{i^B_{B_1,r}} & \mathbb F_r(B)\ar[r]^{\pi^{B_2}_{B,r}} & \mathbb F_r(B_2)\ar[r]& 0}.$$ 

2. If $H_r(X)$ is equipped with a scalar product then 
$${\bf H}_r(B_1) \perp {\bf H}_r(B_2)$$ and
$${\bf H}_r(B)= {\bf H}_r(B_1) \oplus {\bf H}_r(B_2).$$
\end{proposition}

\begin{proposition}\label{P2}
Let $a'<a,$ $b<b'< b''$ and $B_1,$ $B_2$ and $B$  the boxes $B_1= (a',a]\times [b',b''),$ $B_2= (a',a]\times [b,b')$
and $B= (a',a]\times [b,b'') $ (see Figure 3). 

1. The inclusions $B_1\subset B$ and $B_2\subset B$ induce the linear maps 
\begin{equation} i^B_{B_1,r}:  \mathbb F_r(B_1)\to  \mathbb F_r(B)\end{equation} and \begin{equation} \pi^{B_2}_{B,r}:  \mathbb F_r(B)\to  \mathbb F_r(B_2)\end{equation}  such that the following sequence is exact
$$\xymatrix{0\ar [r]& \mathbb F_r(B_1)\ar[r]^{i^B_{B_1,r}} & \mathbb F_r(B)\ar[r]^{\pi^B_{B_2,r}} & \mathbb F_r(B_2)\ar[r]& 0}.$$ 

2. If $\kappa= \mathbb R$ or $\mathbb C$ and $H_r(X)$ is equipped with a scalar product then 
$${\bf H}_r(B_1) \perp {\bf H}_r(B_2)$$ and
$${\bf H}_r(B)= {\bf H}_r(B_1) \oplus {\bf H}_r(B_2).$$ 
\end{proposition}

\mypicture {}

\vskip .2in

\begin{proof} Item 1.  in both Propositions (\ref{P1}) and (\ref{P2})  follows from Observation \ref{O31}  items 1. and 2.
To conclude  item 2. note that ${\bf H}_r(B_2)$  as a subspace of $\mathbb F_r(a'',b)$ in Proposition \ref{P1} and  as a subspace of $\mathbb F_r(a,b'')$  in Proposition \ref{P2} is orthogonal to a subspace  which contains ${\bf H}_r(B_1).$
\end{proof}

In view of Propositions (\ref{P1}) and (\ref{P2}) above one has the following Observation. 
\begin{obs}\label {O35}\

\begin{enumerate}
\item If $B'$and $B''$ are two boxes with $B'\subseteq B''$ and  $B'$ is located in the upper left corner of $B''$ (see picture Figure 4) then the inclusion induces the canonical injective linear maps $i^{B''}_{B',r}: \mathbb F_r(B')\to \mathbb F_r(B'').$ 
\item If $B'$and $B''$ are two boxes with $B'\subseteq B''$ and  $B'$ is located in the lower right  corner of $B''$ (see picture Figure 5) then the inclusion induces the canonical surjective  linear maps $\pi^{B'}_{B'',r}: \mathbb F_r(B'')\to \mathbb F_r(B').$ 
\item If $B$ is a finite disjoint union of boxes $B=\sqcup B_i$ then $\mathbb F_r(B)$ is isomorphic to $\oplus _i \mathbb F_r(B_i);$
the isomorphism is not canonical. 
\item If in addition $\kappa= \mathbb R$ or $\mathbb C$ and $H_r(X)$ is a Hilbert space then 
${\bf H}_r(B)= \oplus_i {\bf H}_r(B_i).$
\end{enumerate}
\end{obs}
\vskip .2in

\mynewpicture { }

In view of the above observation denote by $B(a,b:\epsilon)= (a-\epsilon,a]\times [b, b+\epsilon)$ and define 

$$\boxed{ \hat {\delta}^f_r(a,b):=\varinjlim_{\epsilon \to 0}  \mathbb F_r(B(a,b;\epsilon)).}$$

The limit refers to the direct system $\mathbb F_r(B(a,b; \epsilon'))\to \mathbb F_r(B(a,b;\epsilon''))$ whose arrows are the surjective linear maps induced by  the inclusion of $B(a,b;\epsilon')$ as the lower right corner of $B(a,b;\epsilon'')$ for $\epsilon'<\epsilon''.$

Define also
$$\boxed{\delta^f_r(a,b):=\lim_{\epsilon \to 0}  F_r(B(a,b;\epsilon)).}$$ 
Clearly one has $\dim \hat \delta^f_r(a,b)= \delta^f_r(a,b).$
Denote by $\supp\ \delta^f_r$ the set  $$\boxed{\supp \ \delta^f_r:=\{(a,b)\in \mathbb R^2\mid \delta^f_r(a,b)\ne 0\}.}$$

\begin{obs} \label {O}\
 For any $(a,b)$, $a,b\in \mathbb R,$  the direct system stabilizes and  $\hat \delta_r^f(a,b)= \mathbb F^f(B(a,b;\epsilon))$ for some $\epsilon $ small enough.  Moreover $\delta^f_r(a,b)\ne 0$ implies that $a,b\in CR(f).$  In particular $\supp\  \delta^f_r$ is a discrete subset of $\mathbb R^2.$  
If $f$ is homologically tame then for any $(a,b), a,b\in CR(f),$  $\hat \delta_r^f(a,b)= \mathbb F^f(B(a,b;\epsilon))$ for any $\epsilon,$ $0<\epsilon <\tilde \epsilon (f).$
\end{obs}

\vskip .1in
Recall that for a box $B= (a',a]\times [b,b')$  we have denoted by $\pi^B_{ab,r}: \mathbb F_r(a,b)\to \mathbb F_r(B)$  the canonical projection on $\mathbb F_r(B)= \mathbb F(a,b)/ \mathbb F'(B)$  and for $B'= (a'',a]\times [b,b''),$  $a''\leq a'<a, \  \ b''\geq b'>b,$  we have denoted by $\pi^B_{B',r}: \mathbb F_r(B')\to \mathbb F_r(B)$ the canonical surjective linear map between  quotients spaces induced by $\mathbb F'(B')\subset \mathbb F'(B) \subset \mathbb F(a,b)$. Clearly $$\pi^B_{ab,r}= \pi^B_{B',r}\cdot \pi^{B'}_{ab,r}.$$

Consider  the surjective linear map 
$$\pi_r(a,b): \mathbb F(a,b)\to \varinjlim_{\epsilon \to 0} \mathbb F(B(a,b;\epsilon))=\hat \delta^f_r(a,b)$$
with  $$\pi_r(a,b):= \varinjlim_{\epsilon \to 0} \pi^{B(a,b;\epsilon)}_{ab,r}.$$

\begin{definition}
A {\it special splitting} is a linear map $$s_r(a,b): \hat \delta^f_r(a,b)\to \mathbb F_r(a,b)$$  which 
 satisfies $\pi_r(a,b)\cdot s_r(a,b)= id.$  In particular, in view of Observation \ref{O31}, for any $\alpha>a, \ \beta< b$ we have $\img(s_r(a,b))\subset \mathbb F_r(\alpha, \beta).$ 
\end{definition}
One denotes by $i_r(a,b)$ the composition of $s_r(a,b)$ with the inclusion $\mathbb F_r(a,b)\subset H_r(X)$

The following diagram reviews for the reader the linear maps considered so far. 
In this diagram 
suppose $B=(\alpha',\alpha]\times [\beta, \beta')$ with $a\in (\alpha',\alpha]$ and $b\in [\beta, \beta')$ and $B= B_1 \sqcup B_2$ as in Figure 2 or Figure 3.

\begin{equation}\label {D111}
\xymatrix{ H_r( X) & \mathbb F_r(a,b)\ar[d]_{\pi^B_{ab,r}}\ar[l]_\supseteq \ar[r]_{\pi_r(a,b)} &\hat \delta^{\tilde f}_r(a,b)\ar@/_1pc/[l]_{s_r(a,b)}\ar@/_2 pc/[ll]_{i_r(a,b)}   \ar[ld]^{i^B_r(a,b)}\\
\mathbb F_r(B_1)\ar[r]^{i^B_{B',r}}&\mathbb F_r(B)\ar[r]_{\pi^{B_2}_{B,r}}&\mathbb F_r(B_2).}
\end{equation}
Observe that if $B=B_1\sqcup B_2$ as in Figure 2 or Figure 3, in view of Observations \ref{O35} and \ref{O},  one has 
\begin{obs}\label {O36}\
\begin{enumerate}
\item If $(a,b)\in B_2$  then $\pi^{B_2}_{B,r}\cdot  i^B_r(a,b)$ is injective. 
\item If $(a,b)\in B_1$  then $\pi^{B_2}_{B,r}\cdot  i^B_r(a,b)$ is zero.
\end{enumerate}
 \end{obs}
\vskip .1in 

Choose special splittings $\{ s_r(a,b)\mid (a,b)\in \supp (\delta^{\tilde f}_r)\},$ 
and consider  the sum of $i_r(a,b)'$s for $(a,b)\in \supp (\delta^{\tilde f}_r).$
$$ I_r= \sum _{(a,b)\in \supp (\delta^{\tilde f}_r)} i_r(a,b): \bigoplus_{(a,b)\in \supp (\delta^{\tilde f}_r)} \hat\delta^f_r(a,b) \to H_r(X).$$ 
 and for a finite or infinite box $B$  
the sum 
$$I^B_r= \sum_{(a,b)\in \supp (\delta^{\tilde f}_r) \cap B} i^B_r(a,b): \bigoplus_{(a,b)\in \supp (\delta^{\tilde f}_r) \cap B} \hat\delta^f_r(a,b) \to \mathbb F_r(B).$$ 

For $\Sigma\subseteq \supp (\delta^f_r)$ denote by $I_r(\Sigma)$ the restriction of $I_r$ to $\bigoplus_{(a,b)\in \Sigma} \hat\delta^f_r(a,b) $ and for $\Sigma \subseteq  \supp (\delta^f_r) \cap B$  denote by  $I^B_r(\Sigma)$ the restriction of $I^B_r$ to $\bigoplus_{(a,b)\in \Sigma} \hat \delta^f_r(a,b).$
Note that: 
\begin{obs} \label {O37}\ 

For $B=B_1\sqcup B_2$ as in Figures 2 or Figure 3 and $\Sigma\subseteq \supp\ \delta ^{\tilde f}_r $ with $\Sigma= \Sigma_1 \sqcup \Sigma _2$, $\Sigma_1 \subseteq B_1, \Sigma_2\subseteq B_2$ the  diagram 
$$\xymatrix{ \mathbb F_r(B_1) \ar[r] & \mathbb F_r(B) \ar[r] &\mathbb F_r(B_2)\\
\bigoplus _{(a,b)\in \Sigma_1}\hat \delta^{\tilde f}_r(a,b) \ar[u]_{I^{B_1}_r(\Sigma_1)} \ar[r]&\bigoplus_{(a,b)\in \Sigma} \hat \delta^{\tilde f}_r(a,b) \ar[u]_{I^{B}_r(\Sigma)}\ar[r] &\bigoplus_{(a,b)\in \Sigma_2} \hat \delta^{\tilde f}_r(a,b) \ar[u]_{I^{B_2}_r(\Sigma_2)}}
$$
is commutative.  In particular if $I^{B_1}_r(\Sigma_1) $ and $I^{B_2}_r(\Sigma_2)$ are injective then so is $I^B_r(\Sigma).$
\end{obs}

If $\kappa= \mathbb R$ or $\mathbb C$ and $H_r(X)$ is equipped with a Hilbert space structure, then the inverse of the restriction of $\pi_r(a,b)$ to the orthogonal complement of $\ker(\pi_r(a,b))$ provides a {\it canonical special splitting}. For the canonical special splitting one denotes by 
$\hat{\hat \delta}^f_r$ the and consider the assignment $$\hat{\hat\delta}^f_r (a,b)= {\bf H}_r(a,b):= \img\ s_r(a,b).$$
Then if $X$ is compact in view of Observation \ref{O35} item 4. the assignment $\hat{\hat\delta}^f_r$ is a configuration $\mathcal C^O_{H_r(X)}(\mathbb R^2).$ 
The  configuration 
$\hat {\hat \delta}^f_r (a,b)$ 
has  the  configuration 
$\delta^f_r\in \mathbb C^{\dim H_r(X)}$ as its dimension. 

\vskip .1in
Let $f$ be a map and for  any $(a,b)\in \mathbb R^2$ 
choose a special splitting $s_r(a,b): \hat \delta^f_r(a,b)\to H_r(X).$

\begin{obs}\label {P24}\

\begin{enumerate}
\item For any $\Sigma \subseteq \supp (\delta^f_r)$ resp. $\Sigma\subseteq \supp (\delta^f_r)\cap B$ the linear maps $I_r(\Sigma)$ resp. 
$I^B_r(\Sigma)$ are injective.
\item For any  box $B= (a',a]\times[b,b') 
$  the set $ \delta^f_r\cap B$ is  finite. 
\item  For any box $B,$  the linear map $I^B_r$ is an isomorphism.  
\item If  $X$ compact, $m< \inf f$ and $M> \sup f$ then $H_r(X)= \mathbb F_r((m,M]\times [m, M))$  and $I_r$ is an isomorphism. Therefore  for any special splittings  the collection of subspaces $\img(i_r(a,b))$ provide a configuration, of subspaces of $H_r(X)$ hence and element in $\mathcal C_{H_r(X)}(\mathbb R^2).$
\end{enumerate}
\end{obs} 

\begin {proof} 
Item 1.: If $\Sigma \subset B$ then in view of  Observations \ref{O36} and \ref{O37}  the injectivity  of $I^B_r(\Sigma)$ implies the infectivity of of $I^{B'}_r(\Sigma)$ for any box $B'\supseteq B$ as well as the injectivity of $I_r(\Sigma).$
To check the injectivity of $I^B_r(\Sigma)$ one proceed as follows:

\begin{itemize}
\item If  cardinality of $\Sigma$ is one, then the statement follows from Observation\ref{O36}.
\item If  all elements of $\Sigma,$ $(\alpha_i,\beta_i)$ $i=1,\cdots k$ have the same first component $\alpha_i=a$ the statement follow by induction on $k.$  One writes the box $B= B_1\sqcup B_2$  as in Figure 2 such the $B_2$ contains one element of $\Sigma$, say $(\alpha_1,\beta_1)$  and $B_1$ contains the remaining $(k-1)$ elements. The injectivity follows from Observation\ref{O37} in view of injectivity of $I^{B_2}_r(\Sigma\cap B_2)$ and of $I^{B_1}_r(\Sigma\cap B_1),$ assumed by  the induction hypothesis. 
\item In general one writes  $\Sigma$ as the disjoint union $\Sigma= \Sigma_1\sqcup \Sigma_2\sqcup \cdots \sqcup \Sigma_k$ such that each $\Sigma _i$ contains  all points of $\Sigma$ with the same first component $a_i,$ and $a_k >a_{k-1}\cdots a_2 >a_1.$  One proceeds again by induction on $k.$
One decomposes the box $B$ as in Figure 3,  $B= B_1\sqcup B_2$  such that  $\Sigma _1\subset B_2$ and $(\Sigma \setminus \Sigma_1) \subset B_1$  The injectivity of $I^B_r(\Sigma)$ follows then using Observation \ref{O37}  from the injectivity of $I^{B_2}_r(\Sigma_1)$ and the induction hypothesis which assumes the injectivity of $I^{B_1}_r(\Sigma\cap B_1).$  
\end{itemize}

Item 2. : In view of Item 1. any subset of $\supp (\delta^f_r)\cap B$ with $B=(a',a]\times [b,b')$ has cardinality smaller than $\dim \mathbb F_r(a,b)$ which by 
Proposition \ref {P32} is finite . Hence $\Sigma$ is finite. 

Item 3.: The injectivity of $I^B_r$ is insured by Item 1. The surjectivity follows from the equality of the dimension of the source and of the target 
implied by Observations \ref{O35} and \ref{O}.

Item 4.: follows from definitions and Item 3.

\end{proof}

In case $X$ is not compact for the needs of part II of this paper it is useful to  extend   item 3. of Proposition \ref{P24} to the case of infinite box,  precisely $B(a,b; \infty):= (-\infty,a]\times [b,\infty)$ and evaluate the image of $I_r$ which might not be a finite dimensional space.
For this purpose introduce 
\begin{enumerate}
\item $\mathbb I^f_{-\infty}(r)= \cap _{a\in \mathbb R} \mathbb I^f_a(r)$ and $\mathbb I_f^{\infty}(r)= \cap _{b\in \mathbb R} \mathbb I_f^b(r),$
\item $\mathbb F^f_r(-\infty,b):= \mathbb I^f_{-\infty}(r) \cap \mathbb I^b_f(r)$ and $\mathbb F^f_r(a,\infty):=  \mathbb I^f_a(r) \cap \mathbb I^\infty_f(r),$ 
\item $(\mathbb F^f)'_r (B(a,b;\infty)):= \mathbb F^f_r(-\infty,b) + \mathbb F^f_r(a,\infty),$
\item $\mathbb F^f_r(B(a,b;\infty)): = \mathbb F^f_r(a,b) / (\mathbb F^f)'_r(B(a,b;\infty)).$

\end{enumerate}

\begin{obs} \label{O38} \
\begin{enumerate}
\item In view of finite dimensionality  of $\mathbb F_r(a,b)$  one has:
\begin{enumerate}
\item for any $a$ there exists $b(a)$ such that 
$$\mathbb F_r(a, b(a))= \mathbb F_r(a, b')= \mathbb F_r(a,\infty)$$
provided that  $b'\geq b(a)$
\item for any $b$ there exists $a(b))$ such that 
$$\mathbb F_r(-\infty,b)=  \mathbb F_r(a',b)= \mathbb F_r(a(b),b)$$
provided that $a'\leq a(b).$ 
\end{enumerate}
\item 
In view of item (1.), for $a'< a(b)$ and $b'>b(a),$  the canonical projections
$$ \mathbb F_r (B(a,b;\infty))    \to \mathbb F_r ((a',a]\times [b,b')) \to \mathbb F_r((a(b),a]\times [b,b(a)))$$
are isomorphisms 
\end{enumerate}
\end{obs}

\begin{obs} \label {P39} (Addendum to Observation \ref{P24} item 3.)

 The  maps    $$\oplus_{(a',b')\in \supp (\delta^f_r) \cap B(a,b;\infty)}    i^{B(a,b;\infty)}_r(a',b'): \oplus_{(a',b')\in \supp \delta^f_r\cap B(a,b;\infty)} \ \hat \delta ^f_r (a',b') \to \mathbb F_r(B(a,b;\infty))$$
and
$$\oplus_{(a,b)\in  \supp (\delta^f_r)} \ i_r(a,b): \oplus_{(a,b)\in  \supp (\delta^f_r)}\ \hat \delta ^f_r (a,b) \to H_r(X)/ (\mathbb I^f_{-\infty}(r) + \mathbb I^\infty_f(r))$$
are isomorphisms.
\end{obs}
\begin{proof} 
First isomorphism follows from Observations  \ref{P24} and  \ref{O38}. 
 
 For the second, note that for $k <k'$  (for simplicity in writing we drop $f$ and $r$ from notation)
 $$(\mathbb I_{-\infty} \cap \mathbb I^{-k'}  + \mathbb I_{k'}\cap \mathbb I^\infty) \cap  \mathbb I^{-k}\cap \mathbb I_k= \mathbb I_{-\infty} \cap \mathbb I^{-k}+ \mathbb I_{k}\cap \mathbb I^\infty$$
and  that  $$H_r(X)= \varinjlim_{k\to \infty} \mathbb F_r(k, -k)= \varinjlim_{k\to \infty}\mathbb I^{-k}= \varinjlim_{k\to \infty}\mathbb I_k.$$  
Then in view of stabilization properties  
$$\varinjlim \frac{ \mathbb F(k,-k)} {\mathbb I_{-\infty} \cap \mathbb I^{-k}+ \mathbb I_{k}\cap \mathbb I^\infty}= \frac{H_r(X)}{\mathbb I_{-\infty } +\mathbb I^\infty}$$
\end{proof}
\vskip .2in
Let $D(a,b;\epsilon):= (a-\epsilon, a+\epsilon ]\times [b-\epsilon, b+\epsilon).$ If $x=(a,b)$ one also writes $D(x;\epsilon)$ for 
$D(a,b;\epsilon)$ 

\begin{proposition} \label {P08} (cf \cite{BH} Proposition 5.6)

Let $f:X\to \mathbb R$ be a  tame map and $\epsilon <\epsilon(f)/3.$ For any  map $g:X\to \mathbb R$ which satisfies  $|| f- g ||_\infty <\epsilon$ and $a,b\in Cr(f)$ critical values  one has:

\begin{equation}\label{E3}
 \quad   \sum_{x\in D(a,b;2\epsilon)} \delta^g_r(x)=  \delta ^f_r(a,b), 
\end{equation}
\begin{equation} \label{E40}
 \quad \quad \supp \ \delta^{g}_r\subset \bigcup  _{(a,b)\in \supp\ \delta^{f}_r} D(a,b;2\epsilon). 
\end{equation}
If in addition $H_r(X)$ is equipped with a Hilbert space structure ($\kappa= \mathbb R$ or $\mathbb C$) the above statement can be strengthen to 

\begin{equation}\label{E5}
x\in D(a,b;2\epsilon)\Rightarrow \hat \delta^g_r(x)\subseteq \hat \delta ^f_r(a,b), \ 
\quad   \oplus_{x\in D(a,b;2\epsilon)}\hat \delta^g_r(x)= \hat \delta ^f_r(a,b) 
\end{equation}
\end{proposition}

Proposition (\ref{P08}) implies  that in an $\epsilon-$neighborhood of a  tame map $f$ (w.r. to the $||\cdots ||_\infty$ norm) any other map  $g$ has the support of $\delta^g_r$  in a   $2\epsilon-$neighborhood of  the support of $\delta^f_r$ and in case $X$ compact is  of  cardinality counted with multiplicities equal to $\dim H_r(X).$

\vskip .1in

\noindent {\bf Proof of Proposition (\ref {P08})} (cf \cite{BH}).

\vskip .1in
 
Consider a collection of real numbers $C:=\{ \cdots c_i < c_{i+1} <c_{i+2}\cdots , i\in \mathbb Z\}$ which satisfies the following properties:

\begin{enumerate}
\item $Cr(f)\subseteq C$
\item $c_{i+1}- c_i  >\epsilon (f)$  
\item $\lim_{i\to \infty}  c_i=  \infty $ 
\item $\lim_{i\to -\infty}  c_i= -\infty.$
\end{enumerate} 
 
Next, one establishes two intermediate results, Lemmas \ref{L43} and \ref{P44} below. 

\begin{lemma} \label {L43} \
 For $f$ as in Proposition \ref {P08} and $c_i, c_j\in C$ one has: 
\begin{equation}\label {E023'}
\hat \delta^{f}_r(c_i,c_j)=  \mathbb F^f_r((c_{i-1}, c_i]\times [c_j, c_{j+1}))=  \mathbb F^f_r (c_{i}, c_{j})/ \mathbb F^f _r(c_{i-1}, c_{j}) +  \mathbb F^f_r (c_{i}, c_{j+1})
\end{equation}
and therefore 
\begin{equation}  \label {E023}
\begin{aligned}
\delta^{f}_r(c_i,c_j)=  &F^f_r((c_{i-1}, c_i]\times [c_j, c_{j+1}))= \\ 
F^f_r (c_{i-1}, c_{j+1}) + & F^f_r (c_{i}, c_{j})-  F^f_r (c_{i-1}, c_{j})-  F^f_r (c_{i}, c_{j+1}).
\end{aligned}
\end{equation}

\end{lemma}

\begin{proof}   
It is known, cf \cite {Hu}, that $X$ closed subset of $Y$  and $X,Y$ ANRs imply that $X$ is a neighborhood deformation retract. Then in  view of the tameness of $f$  
for any $0< \epsilon', \epsilon''< \epsilon (f)$ 
one has 
\begin{equation}\label {E4}
\begin{aligned}
\mathbb F^f_r (c_{i}, c_{j})=&  \mathbb F^f_r (c_{i}+\epsilon', c_{j})=  \mathbb F^f_r(c_{i+1}-\epsilon'', c_j)= \mathbb F^f_r(c_{i+1}-\epsilon'', c_{j-1}+\epsilon'')\\
\mathbb F^f_r (c_{i}, c_{j})= & \mathbb F^f_r (c_{i},  c_{j}-\epsilon')=  \mathbb F^f_r(c_{i}, c_{j-1}+\epsilon'')=  \mathbb F^f_r(c_{i+1}-\epsilon' , c_{j-1}+\epsilon''). 
\end{aligned}
\end{equation}
Since $\epsilon <\epsilon (f)$  in view of the definition of $\hat \delta^f_r,$
one has

\begin{equation}\label{E6}
\begin{aligned}
\hat \delta^{f}_r(c_i,c_j)= &\mathbb F^f_r((c_i-\epsilon, c_i]\times [c_j, c_j+\epsilon))= \\
&\mathbb F^f_r (c_{i}, c_{j})/ 
\mathbb F^f_r (c_{i}-\epsilon, c_{j}) + 
\mathbb F^f_r (c_{i}, c_{j}+\epsilon). 
\end{aligned}
\end{equation}
Combining (\ref {E6}) with (\ref{E4}) one obtains  
the equality (\ref {E023'})
\begin{equation*}
\delta^{f}_r(c_i,c_j)= 
\mathbb F^f_r (c_{i}, c_{j})/  \mathbb F^f_r (c_{i-1}, c_{j}) + \mathbb F^f_r (c_{i}, c_{j+1}).
\end{equation*}
Since $\mathbb F^f(c_{i-1}, c_j ) \cap  \mathbb F^f(c_i, c_{j+1})=   \mathbb F^f(c_{i-1}, c_{j+1})$
one has 

$\dim (\mathbb F^f_r (c_{i-1}, c_{j}) + \mathbb F^f_r (c_{i}, c_{j+1}))= \dim \mathbb F^f_r (c_{i-1}, c_{j}) + \dim \mathbb F^f_r (c_{i}, c_{j+1}- \dim \mathbb F^f (c_{i-1}, c_{j+1})$  and  the equality 
(\ref {E023}) follows.

\end{proof}

To  simplify the  notation the index $r$ in  the following Lemma  will be dropped off.

\begin{lemma}\label{P44}\
 
 Suppose $f$ is tame.  
Let $a= c_i, b= c_j,$  $c_i,c_j \in C$   and  $\epsilon <\epsilon(f)/3.$   If $g$ is a continuous map with $|| f- g ||_\infty <\epsilon$ then 

\begin{equation}\label {E8'}
\begin{aligned}
  \mathbb F^g_r(a-2\epsilon, b+2\epsilon) = &\mathbb F^f_r(c_{i-1}, c_{j+1})\\
 \mathbb F^g_r(a+2\epsilon, b-2\epsilon)= &\mathbb F^f_r(c_i,c_j)\\
 \mathbb F^g_r(a+2\epsilon, b+2\epsilon) = &\mathbb F^f_r(c_{i},c_{j+1})\\
 \mathbb F^g_r(a-2\epsilon, b-2\epsilon)= &\mathbb F^f_r(c_{i-1}, c_{j}).
\end{aligned}
\end{equation}
\end{lemma} 

\begin{proof}
Since $||f-g||_\infty<\epsilon,$ in view of Observation \ref{O31} item 3.  one has  

\begin{equation}\label {E9'}
\begin{aligned}
 \mathbb F^f_r(a-3\epsilon, b+3\epsilon)\subseteq& \ \mathbb F^g_r(a-2\epsilon, b+2\epsilon) \subseteq \mathbb F^f_r(a-\epsilon, b+\epsilon),\\
 \mathbb F^f_r(a+\epsilon, b-\epsilon) \subseteq& \mathbb F^g_r(a+2\epsilon, b-2\epsilon) \subseteq \mathbb F^f_r(a+3\epsilon, b-3\epsilon),\\
 \mathbb F^f_r(a+\epsilon, b+3\epsilon) \subseteq& \mathbb F^g_r(a+2\epsilon, b+2\epsilon) \subseteq \mathbb F^f_r(a+3\epsilon, b+\epsilon),\\
 \mathbb F^f_r(a-3\epsilon, b-\epsilon) \subseteq& \mathbb F^g_r(a-2\epsilon, b-2\epsilon) \subseteq \mathbb F^f_r(a-\epsilon, b-3\epsilon).
\end{aligned}
\end{equation}
\vskip .1in

Since $3\epsilon<\epsilon(f)$ one has 

\begin{equation}\label {E9''}
\begin{aligned}
 \mathbb F^f(a-3\epsilon, b+3\epsilon)&=
 \mathbb F^f(a-\epsilon, b+\epsilon) ,\\
 \mathbb F^f(a+\epsilon, b-\epsilon) &=
  \mathbb F^f(a+3\epsilon, b-3\epsilon),\\
 \mathbb F^f(a+\epsilon, b+3\epsilon) &=
  \mathbb F^f(a+3\epsilon, b+\epsilon),\\
 \mathbb F^f(a-3\epsilon, b-\epsilon) &= \mathbb F^f(a-\epsilon, b-3\epsilon).
\end{aligned}
\end{equation}

which imply that in the  equation (\ref{E9'})  the inclusion  "$\subseteq $" is actually the equality ''$=$" .

Note that in view equalities (\ref{E4}) and for  $\epsilon', \epsilon'' <\epsilon(f)$ one has  
 
\begin{equation}\label{E10}
\begin{aligned}
\mathbb F^f(c_{i-1}, c_{j+1})=& \mathbb F^f(a-\epsilon', b+\epsilon'')\\
\mathbb F^f(c_i, c_j)=& \mathbb F^f(a+\epsilon', b-\epsilon'')\\
\mathbb F^f(c_i, c_{j+1})=& \mathbb F^f(a+\epsilon', b+\epsilon'')\\
\mathbb F^f(c_{i-1}, c_j)=&\mathbb F^f(a-\epsilon', b-\epsilon'').
\end{aligned}
\end{equation}

Then (\ref{E9'}) and  (\ref{E10}) imply the equalities (\ref{E8'}) hence the statement of Lemma \ref{P44}.

\end{proof}
\vskip .1in
Next observe that 
 Lemma (\ref {P44}) gives (for $a= c_i,$  $b=c_j$with  $c_i, c_j \in C$) the equality 

$\mathbb F^g((a-2\epsilon, a+2\epsilon]\times [b-2\epsilon, b+2\epsilon))= \mathbb F^f((c_{i-1}, c_i]\times [c_j, c_{j+1})).$ 

This combined with  Lemma  \ref {L43}  implies     
$\mathbb F^g((a-2\epsilon, a+2\epsilon]\times [b-2\epsilon, b+2\epsilon))= \hat \delta^f(a,b),$
which combined with Proposition (\ref {P24}) implies the inclusion (\ref{E3}) and the equality (\ref{E5})  and this not only for critical values but for any $a,b\in C.$

\medskip 

To check inclusion (\ref{E40})  observe  that:
\begin{enumerate}
\item 
 $||f-g||_\infty <\epsilon $ implies 
$X^f_a \subset X^g_{a+\epsilon} \subset X^f_{a+2\epsilon}$ and $X_f^b \subset X_g^{b-\epsilon} \subset X_f^{b-2\epsilon} $
 
and when $a,b \in C$ 
 \begin{equation} \label {E13}
\mathbb F^f(a,b)\subseteq \mathbb F^g(a+\epsilon, b-\epsilon)\subseteq \mathbb F^f(a+2\epsilon, b-2\epsilon).\end{equation}

\item When $\epsilon <\epsilon(f)/3$  inclusions (\ref{E13}) imply 
 \begin{equation*} \label {E14}
 \mathbb F^f(a,b)= \mathbb F^g(a+\epsilon, b-\epsilon)= \mathbb F^f(a+2\epsilon, b-2\epsilon)
 \end{equation*} 
  which in view of  Observation \ref{P39} 
  implies  

\begin{equation}\label{E16}
\begin{aligned}
\sum_{x\in (-\infty,a]\times (b,\infty) \cap \supp \delta^f_r}  \delta^f_r(x) = &\sum_{y\in (-\infty,a+ \epsilon]\times (b-\epsilon,\infty)\cap \supp \delta^g_r}\delta^g_r(y)=\\
&\sum_{x\in (-\infty,a+2\epsilon]\times (b-2\epsilon,\infty)\cap \supp \delta^f_r}  \delta^f_r(x)
   \end{aligned}
   \end{equation}
\end{enumerate}

Since $\mathbb R^2= \cup_{i\in \mathbb Z}  B(c_i, c_{-i};\infty)$
the equalities (\ref{E16}) and equality (\ref {E3}) rule out the existence of  $x\in \supp (\delta^g_r)$
away from $ \cup_{x\in \supp (\delta^f_r)} D(x;  2\epsilon).$
which 
finishes the proof of Proposition \ref {P08}.

\vskip .2in 
 
Let $K$ be a compact ANR and $f:X\to \mathbb R$ be a map. Denote by 
$$\overline f_K; X\times K\to \mathbb R$$  the composition $f\cdot \pi_K$ with $\pi_K: X\times K\to X$ the first factor projection.
If $f$ is weakly tame then so is  $\overline f_K$ and the set of critical values of $f$ and of $\overline f_K$  are the same. 
Moreover in view of the K\"unneth theorem about the homology of the cartesian product  of two spaces  one has:
\begin{obs}\label {O313}\ 
\begin{enumerate}
\item $\mathbb F^{\overline f_K}_r(a,b)= \oplus_{0\leq k\leq r} \mathbb F^f_k(a,b)\otimes H_{r-k}(K)$ and therefore
\item $\hat \delta ^{\overline f_K}_r(a,b)= \oplus_{0\leq k\leq r} \hat \delta^f_k(a,b)\otimes H_{r-k}(K),$ and when $K$ is acyclic
\item $\hat \delta ^{\overline f_K}_r(a,b)= \hat \delta^f_k(a,b).$ 
\end{enumerate} 
\end{obs} 
\vskip .1in
Note  that the embedding $ I: C(X;\mathbb R)\to C(X\times K;\mathbb R)$ defined by $I(f)= \overline f_K$ is an isometry when both spaces are equipped with the distance $||\cdots ||_\infty.$    Note also that when $K$ is acyclic  one has $\delta_r^f= \delta_r^{I(f)}$ and $\hat \delta_r^f= \hat \delta_r^{I(f)}$ provided that $H_r(X)$ is identified  with $H_r(X\times K).$

\section{The main results} \label {R}

\begin{theorem} (Topological results) \label {T1}\ 
Suppose $X$ is compact and $f:X\to \mathbb R$ a map\footnote {this means $X$ also ANR and $f$ continuous}. 
Then
\begin{enumerate}

\item 
$\delta_r^f(x)\ne 0$ with $x=(a,b)$ implies that both $a,b \in CR(f).$  
\item $\sum_{x\in \mathbb R^2} \delta^f_r(x)=\dim H_r(X)$  and \ $\bigoplus_{x\in \mathbb R^2} \hat \delta^f_r(x)=H_r(X).$ In particular 
$\delta^f_r\in \mathcal C_{\dim H_r(X)}(\mathbb R^2),$
\item if $H_r(X)$ is equipped with a Hilbert space structure then  
$\hat{\hat \delta}^f\in \mathcal C^O_{H_r(X)}(\mathbb R^2),$     
\item If $X$ is homeomorphic to a finite simplicial complex or a compact Hilbert cube manifold  then for an open and dense set of maps $f$  in the space of continuous maps with compact open topology $\delta^f_r(x)= 0$ or $1.$
\end{enumerate}
\end{theorem}

The statements 1. and 3.  formulated in terms of barcodes  cf. \cite{BD11},  were verified first in \cite{BH} under the hypothesis of "$f$ a tame map" .

\vskip .1in 
\begin{theorem} (Stability)\label{T2}
Suppose $X$ is a 
compact ANR. 

1. The assignment $ f 
\rightsquigarrow \delta^f_r$ provides a continuous  map from the space of real valued maps $C(X;\mathbb R)$ equipped with the compact open topology to 
the space of configurations $\mathcal C_{b_r}(\mathbb R^2)= \mathbb C^{b_r},\ b_r=\dim H_r(X),$ equipped with the collision topology (also regarded as the space of monic polynomials of degree $b_r$). Moreover, with respect to the canonical metric $\underline D$ on the space of configurations, which induces the collision topology,  one has $$\underline D (\delta^f , \delta^g) < 2 D(f,g).$$  Recall that  $D(f,g):= || f-g||_\infty= sup _{x\in X} |f(x)- g(x)|.$

2. If $\kappa= \mathbb R$ or $\mathbb C$ then the assignment $f\rightsquigarrow \hat{\hat \delta}^f_r$ is continuous w.r. to both collision topologies.
(the continuity w.r. to the first implies with the second).
\end{theorem}
Theorem \ref{T2} 1.was first established in \cite{BH} under the hypothesis $X$ homeomorphic to a finite simplicial complex 
is given in section (\ref{S7}).

\vskip .1in

\begin{theorem}  (Poincar\'e Duality) \label{T3}\

1. Suppose $X$ is a closed smooth manifold \footnote {the result remain probably true  as stated for topological manifolds  based essentially on the same arguments but being unable to find appropriate references we formulate them under the hypothesis of smoothness} of dimension $n$ which is $\kappa-$orientable and $f$ a continuous map.
Then $\delta^f_r(a,b)= \delta^{f}_{n-r}(b,a).$ 

2. In addition any collection of isomorphisms $H_r(X)\to H_r(X)^\ast$  induce the isomorphisms of the configuration $\hat \delta^f_r$ and $\hat\delta^{f}_{n-r}\cdot \tau$ with $\tau (a,b)= (b,a).$
\end{theorem} 
 
Item 1. of the above theorem was established in \cite{BH} for $f$ a tame map.

\subsection{Proof of Theorem   \ref{T1}} 

Items 1. and 2.  and 3. are contained in Observation \ref{P24} and Observation \ref{O}. 

Item 4.  In view of Theorem \ref{T2} whose proof does not involve Theorem \ref {T1} it suffices to establish only the density  in the space of all continuous functions of tame maps $f$ with $\delta^f_r$ taking values only $0$ and $1.$ 

We say that a tame map  $f:X\to \mathbb R$  satisfies  property $G$ if the 
following holds;

{\bf Property G:}
There exists a finite sequence of real numbers  $a= a_0<a_1 <\cdots a_n<a_{n+1}=b$ such that:
\begin{enumerate}\label{E111}
\item $\mathbb I^f_a(r)= 0, \ \mathbb I^f_b(r)= H_r(X),$
\item For any $i\geq 1$ $\dim (\mathbb I^f_{a_i} / \mathbb I^f_{a_{i-1}}) \leq 1.$ 
\end{enumerate}
The verification of item 4 is based on the  observations (\ref {O44}) and (\ref {O45}).
\begin{obs} \label {O44} For any tame map $f$ which satisfies property $G$ the configuration $\delta^f_r$ takes only the values $0$ and $1.$
\end{obs}
If  $f$ has Property $G$ then it satisfies $\dim (\mathbb I^f_{a_i} / \mathbb I^f_{a_{i-1}}) \leq 1$ for $a_i= c_i, i=1, \cdots, n;$   since for  $\alpha < \beta$ with no critical value in the open interval $(\alpha, \beta)$ and $\beta$  a regular value  the inclusion $X^f_\alpha \subset X^f_\beta$ induces isomorphism in homology and for any $a'\leq a\leq b\leq b'$ 
$\dim(\mathbb I^f_b(r) / \mathbb I^f_a(r)) \leq \dim(\mathbb I^f_{b'}(r) / \mathbb I^f_{a'}(r)).$   

If so, then for any two consecutive critical values $c_{i-1}<c_i$ and any other critical value $c_j $ the inclusion $\mathbb F_r(c_{i-1}, c_j)\subseteq \mathbb F_r(c_i, c_j)$ has cokernel of dimension at most one which by (\ref {E023'}) in Lemma \ref {L43} implies that $\delta^f_r$ takes only the 
values $0$ and $1.$
Based on this observation, if 
$X$ is a compact smooth manifold (possibly with boundary), any Morse function $f:X\to \mathbb R$ which takes different values of different critical points has  property $G.$ 

Indeed if $\{\cdots c_i < c_{i+1} <\cdots \}$ is the   collection  of all critical values, $X^f_{c_{i+1}}$ is  homotopy equivalent to a space obtained from $X^f_{c_i}$ by adding a closed disk $D^k$ along $\partial D^k= S^{k-1}$ or $\partial D^k_+= D^{k-1},$ which insures that  Property $G$ is satisfied. Since the set of such Morse functions is dense in the space of all continuous functions equipped with the $C_0-$topology,  Item 4 is verified (once Theorem \ref {T2} is established). 

If $X$ is a  compact Hilbert cube manifold, then is homeomorphic to $M\times Q$ with $M$ a compact smooth manifold (possible with boundary),  and  any continuous map $f:X\to \mathbb R$ is arbitrarily closed to $\overline  f_Q,$ with $f: M\to \mathbb R$ a Morse function. 
This observation establishes Item 4.  for compact Hilbert cube manifolds.

If $X$ is a finite simplicial complex one needs  the following observation.
\begin{obs}\label {O45}
If $X$ is a finite simplicial complex and $a <b$ one can construct a  map $h:X\to \mathbb R$ simplicial on the barycentric subdivision of $X$ with the following properties:
\begin{enumerate}
\item $a< h(x)<b,$   
\item $h$ takes different values on the barycenters of different simplices, 
\item the value of $h$ on the barycenter of a simplex $\sigma$ is strictly larger than  the values of $h$ on the barycenter of any of its faces.
\end{enumerate} 
\end{obs}
The construction is straightforward. Such map satisfies Property $G$ since adding a simplex to a finite simplicial complex might change the dimension of the homology with at most one unit and for any $\alpha$ $X^h_\alpha$ retracts by deformation to the simplicial complex generated by the barycenters  on which $h$ takes value smaller or equal to $\alpha.$

For $f: X\to \mathbb R$ a simplicial map, $X$ finite simplicial complex with critical values $\{ \cdots c_{i-1} < c_i\cdots\}$ in case that  for some $i$ $\dim (\mathbb I^f_{c_i}/ \mathbb I^f_{c_i}) \geq 2$ on chooses $\epsilon  <\epsilon(f)/2$ and  a subdivision of $X$  which makes $f^{-1}(c_i\pm \epsilon/2), f^{-1}(c_i))$  and then $f^{-1}([c_i- \epsilon/2, c_i+\epsilon/2]), f^{-1}([c_i, c_i+\epsilon])$ subcomplexes. One takes  the barycentric subdivision of this subdivision  and one replaces replaces $f $ by $g,$ the simplicial map for the new triangulation. We define the map $g$ to take the same value as $f$ on the barycenters of simplices not contained in $f^{-1}(c_i)$ and as $h$ constructed using the previous observation \ref {O45} for $a= c_I-\epsilon/2, b= c_i+\epsilon /2$ on the barycenters of simplices contained in $f^{-1}(c_i).$ 
The map $g$ gets as possible critical values, in addition to the critical values of $f$ the critical values of $h= g|_{f^{-1}(c_i)}.$  We leave the reader to check that $g$ satisfies property $G$ in view of the fact that $h$ does and $\epsilon <\epsilon(f).$  Clearly  $g$ differs from $f$ by less that $\epsilon$ as if follows from construction.  

Since simplicial maps (for some subdivision) are dense in the space of continuous maps  and any simplicial map is arbitrarily closed to one which satisfies Property $G,$ Item 4 follows.  
q.e.d

\vskip .2in
\subsection {Stability (Proof of Theorem \ref{T2})}\label {S7}
Stability theorem is a consequence of  Proposition \ref{P08}. In order to explain this we  begin with a few observations. 
\begin{enumerate}
\item  Consider  the space  of  maps  $C(X, \mathbb R),$  $X$ a  compact ANR,  equipped with the compact open topology which is  induced from the metric   
$D(f,g):= \sup_{x\in X} |f(x)- g(x)|= ||f-g||_\infty.$   This metric  is complete. 

\item  Observe that if $f,g\in  C(X,\mathbb R)$  then  for any $t\in [0,1]$    $h_t:=
 t f(x) +(1-t) g(x)\in C(X;\mathbb R)$ is continuous 
 and for any $0=t_0 <t_1 \cdots t_{N-1} <t_N=1$ one has the inequality
\begin{equation}\label {E0}
D(f,g)= \sum _{0\leq i <N} D(h_{t_{i+1}}, h_{t_i}).
\end{equation} 

\item If  $X$ is a simplicial complex let  $\mathcal U\subset  C (X,\mathbb R)$ denote the subset of p.l. maps.   Then: 

i. $\mathcal U$ is a dense subset in $C(X,\mathbb R),$

ii. if $f,g\in \mathcal U$ then  $h_t \in \mathcal U,$ hence  $ \epsilon(h_t) >0,$ hence for any $t\in[0,1]$ there exists $\delta (t)>0$ s.t.
$t',t''\in (t-\delta(t), t+\delta(t))$ implies 
$D(h_{t'}, h_t)<\epsilon (h_t)/3.$

These two statements are not hard to check.  Recall that:
 \newline  \ -  $f$ is p.l. on $X$  if with respect to some subdivision of $X$ $f$ is simplicial (i.e.  the restriction of $f$ to each simplex is  linear) and 
 \newline \ -  for  any two  p.l. maps $f,g$ there exists a common subdivision of $X$ which makes $f$ and $g$ simultaneously  simplicial, hence  $h_t$ is a simplicial map for any $t.$  
\newline Item (i.) follows from the fact that  continuous maps can be approximated with arbitrary accuracy by p.l. maps and item (ii.) follows from the  continuity in $t$ of the family $h_t$ and from the compacity of $X.$ 

\item Consider $\mathcal C_{b_r}(\mathbb R^2) = \mathbb C^{b_r},$ $b_r= \dim(H_r(X)$, with the canonical metric$\underline D$ which is  complete.
Since any map in $\mathcal U$ is tame,  in view 
 Proposition (\ref {P08}),  $f, g \in \mathcal U$ with $D(f,g) <\epsilon(f)/3$ imply
\begin{equation}\label {E14}
\underline D(\delta^f_r, \delta^g_r) \leq 2 D(f,g).
\end{equation} 
\end{enumerate}

To prove Theorem \ref{T2} first check that the inequality (\ref{E14}) extends to all $f,g \in \mathcal U.$   To do that we
 start with $f,g\in \mathcal U$   and consider the homotopy  $h_t,$ $t\in [0,1]$ defined above.  
\vskip .1in 
Choose a sequence $0<t_1< t_3 <t_5,\cdots t_{2N-1} <1$  
such that 
for $i=1,\cdots, (2N-1)$ the intervals $(t_{2i-1}-\delta(t_{2i-1}), t_{2i-1}+\delta(t_{2i-1}))$ cover $[0,1]$ and 
$(t_{2i-1},t_{2i-1}+\delta(t_{2i-1}))\cap  (t_{2i+1}-\delta(t_{2i+1}), t_{2i+1})\ne \emptyset.$   This is possible in view of the compacity  of $[0,1].$ 

Take $t_0=0, t_{2N}=1$ and $t_{2i}\in (t_{2i-1},t_{2i-1}+\delta(t_{2i-1}))\cap  (t_{2i+1}-\delta(t_{2i+1}).$ 
To simplify the notation abbreviate  $ h_{t_i}$ to $h_i.$

In view of item 3. ii. and item 4. (inequality (\ref{E14}) above  one has: 

$|t_{2i-1}- t_{2i}|<\delta(t_{2i-1})$ implies  $\underline D(\delta^{h_{2i-1}},\delta^{h_{2i}})< 2D(h_{2i-1}, h_{2i})$ and 

$|t_{2i}- t_{2i+1}|<\delta(t_{2i+1})$ implies  $\underline D(\delta^{h_{2i}},\delta^{h_{2i+1}} )< 2D(h_{2i}, h_{2i+1})$
\newline Then we have $$\underline D(\delta^f, \delta^g) \leq \sum_{0\leq i<2N-1} \underline D(\delta^{h_{i}}, \delta^{h_{i+1}})  \leq 2\sum_{0\leq i<2N-1}  D(h_i, h_{i+1})=  D(f,g).$$ 

In view of the density of $\mathcal U$  and the completeness of the metrics on $C(X;\mathbb R)$ and $\mathcal C_{b_r}(\mathbb R^2)$ the inequality (\ref {E14}) extends to the entire $C(X;\mathbb R)$ in case  $X$ is a simplicial complex. Indeed the assignment $\mathcal U\ni f\rightsquigarrow \delta^f_r\in C_{b_r}(\mathbb R^2)$ preserve the Cauchy sequences.

Next we verify the inequality (\ref {E14}) for $X= K\times Q,$ $K$ simplicial complex and $Q$ the Hilbert cube.

For this purpose we write $Q:= I^k\times Q^{\infty-k}$ 
and say that $f: K\times Q\to \mathbb R$ is a $(\infty-k)-$p.l. map if $f= \overline g_{Q^{\infty-k}}$ (see subsection \ref{SS23} for the definition of 
$\overline  g_{Q^{\infty-k}}$) with $g: K\times I^k\to \mathbb R$ a p.l. map.  Clearly a $(\infty-k)-$p.l. map is a $(\infty-k')-$p.l. map for $k'\geq k.$ 

Denote by $C_{p.l.}(K\times Q;\mathbb R)$ the set of maps in $C(K\times Q;\mathbb R)$ which are $(\infty-k)-$ p.l.  for some $k.$

In view of 
Observation \ref{O22}\   $C_{p.l}(K\times Q;\mathbb R)$  is dense  in $C(K\times Q;\mathbb R).$ 
To conclude that (\ref {E14}) holds for $K\times Q,$  it suffices to check the inequality for  $f_1= \overline (g_1)_{Q^{\infty-k}},  f_2 =\overline (g_2)_{Q^{\infty-k}}\in C_{p.l}(K\times Q;\mathbb R).$  The inequality  holds since, in view of Observation \ref {O313},  we have $\delta^{f_i}= \delta^{g_i}.$ 

Since by Theorem \ref{T23} any compact Hilbert cube manifold is homeomorphic to $K\times Q$ for some finite simplicial complex $K$, the inequality (\ref {E14})
 holds for $X$ any compact  Hilbert cube manifold. Since for any $X$ a compact ANR, by Theorem \ref{T23},\ 
$X\times Q$ is a Hilbert cube manifold, $I: C(X;\mathbb R)\to C(X\times Q;\mathbb R)$ defined by $I(f)= \overline f_Q$ is an isometric embedding and $\delta^f= \delta^{\overline f_Q},$ the inequality (\ref{E14}) holds for any 
 $X$ a compact ANR.

Both parts 1 and 2 of Theorem \ref{T2} follow from inequality (\ref{E14}) and  Proposition \ref{P08}  (\ref{E5}).   

\vskip .2in 
\subsection{Poincar\'e Duality (Proof of Theorem \ref{T3})}\label{S5}
Before we proceed to the proof of Theorem \ref{T3} the following elementary observation on linear algebra used also in part II will be useful. 
\vskip .1in

For the commutative diagram $$
E:= \begin{cases}\xymatrix{ C\ar[r]^{\gamma_2}\ar[d]^{\gamma_1}&A_2\ar[d]_{\alpha_2}\\A_1\ar[r]^{\alpha_1}&B}\end{cases}$$ 

denote by 
\begin{equation*}
\begin{aligned}
 \ker (E):= &\ker (\xymatrix {C\ar[r]^\gamma &A_1\times_B A_2})\\
\coker (E):=&\coker(\xymatrix {A_1 \oplus_C A_2 \ar[r]^\alpha&B})
\end{aligned}
\end{equation*} 
with 

\begin{equation*}
\begin{aligned}
A_1\times_B A_2=&\{ (a_1, a_2)\in A_1\times A_2 \mid \alpha_1(a_1)= \alpha_2(a_2)\}\\
A_1\oplus_C A_2 =& A_1 \oplus A_2  /  \{ (a_1, a_2)\in A_1\times A_2 \mid a_1= \beta_1(c), a_2= -\beta_2(c) \ \rm{for \ some }\ c\in C\}
\end{aligned}
\end{equation*} 
and with  $\gamma (c)= (\gamma_1(c), \gamma_2(c))$  and $\alpha(a_1, a_2)= \alpha_1(a_1) + \alpha_2(a_2).$ 

If one denotes by $E^\ast$ the dual diagram 
$$
E^\ast:= \begin{cases}\xymatrix{ C^\ast &A^\ast _2\ar[l]_{\gamma_2^\ast}\\A^\ast_1\ar[u]^{\gamma^\ast_1}&B^\ast \ar[l]^{\alpha^\ast_1}\ar[u]^{\alpha^\ast _2}}\end{cases}$$ 
then we have a canonical isomorphism
\begin{equation}\label {EEE1}
\ker(E)= (\coker (E^\ast))^\ast.
\end{equation}
Note that
\begin{proposition} \label {P99}\ 

1. If in the diagram $E$ all arrows are injective and $\alpha$  is injective then 
$\dim (\coker E)= \dim C + \dim B  -\dim A_1 -\dim A_2.$

2. If in the diagram $E$ all arrows are surjective  and $\gamma$  is surjective then 
$\dim (\coker E)= \dim C + \dim B  -\dim A_1 -\dim A_2.$
\end{proposition}
The proof is a straightforward calculation of dimensions.
\vskip .1in

For the proof of extended Poincar\'e Duality claimed by Theorem \ref{T3} it is useful to provide an alternative definition of $\mathbb F_r(B)$ for a box $B.$

For this purpose introduce the quotient space $$\mathbb G_r(a,b)= H_r(X)/ \mathbb I_a(r) +\mathbb I^b(r).$$ 

Consider a box $B= (a',a]\times [b, b')$  and denote by $\mathcal G(B)$ and $\mathcal F(B)$ the diagrams 

$$\mathcal G(B):= \xymatrix {\mathbb G_r(a',b')\ar[r]\ar[d]& \mathbb G_r(a,b')\ar[d]\\\mathbb G_r(a',b)\ar[r]& \mathbb G_r(a,b)}\  \mathcal F(B):= \xymatrix {\mathbb F_r(a',b')\ar[r]\ar[d]& \mathbb F_r(a,b')\ar[d]\\
\mathbb F_r(a',b)\ar[r]& \mathbb F_r(a,b)}$$
whose   arrows  are induced by the inclusions $ \mathbb I_{a'}(r)\subseteq \mathbb I_a(r)$ and $\mathbb I^{b'}(r)\subseteq \mathbb I^b(r).$
Introduce $$\mathbb G^f_r(B):= \ker \mathcal G(B)$$  and recognize  that $$\mathbb F^f_r(B)= \coker \mathcal F(B).$$ 

Note that  the hypotheses of Proposition \ref{P99} are verified, (1) for  $\mathcal G(B)$ and (2) for $\mathcal F(B),$ and 
$\mathbb G_r(B)$  identifies to $\ker (\mathcal G(B))$ and $\mathbb F_r(B)$ to $\coker (\mathcal F(B)).$ 

Since  $\mathbb G_r(a',b)\times _{\mathbb G_r(a,b)}\mathbb G_r(a,b') = H_r(X)/ ((\mathbb I_{a'}(r)  + \mathbb I^{b}(r) )\cap (\mathbb I_{a}(r)  + \mathbb I^{b'}(r) )),$  the vector space $\mathbb G_r(B)$ is 
  canonically isomorphic to 
\begin{equation}
\boxed{\label {E100}(\mathbb I_{a'}(r) + \mathbb I^{b}(r))\cap (\mathbb I_{a}(r) + \mathbb I^{b'}(r)))/ (\mathbb I_{a'}(r) +\mathbb I^{b'}(r)).}
\end{equation}

Similarly since $\mathbb F_r(a',b) \oplus_{\mathbb F_r(a',b')} \mathbb F_r(a,b'))= (\mathbb I_{a'}(r)\cap \mathbb I^b(r) +  \mathbb I_{a}(r)\cap \mathbb I^{b'}(r)),$ the vector space $\mathbb F_r(B)$ is canonically isomorphic to $\boxed{\mathbb I_a(r)\cap \mathbb I^b(r) / (\mathbb I_{a'}(r)\cap \mathbb I^b(r) +  \mathbb I_{a}(r)\cap \mathbb I^{b'}(r))}.$

 The obvious inclusion 
$\mathbb I_a(r)\cap \mathbb I^b(r) \subseteq (\mathbb I_{a'}(r) + \mathbb I^{b}(r))\cap (\mathbb I_{a}(r) + \mathbb I^{b'}(r))$ induces the  linear map 
$$\mathbb F_r(B)= \mathbb I_a(r)\cap \mathbb I^b(r) / (\mathbb I_{a'}(r)\cap \mathbb I^b(r) + \mathbb I_{a}(r)\cap \mathbb I^{b'}(r))  \to  (\mathbb I_{a'}(r) + \mathbb I^{b}(r))\cap (\mathbb I_{a}(r) + \mathbb I^{b'}(r))/ (\mathbb I_{a'} (r)+ \mathbb I^{b'}(r)= \mathbb G_r(B) .$$

\begin{proposition}\label {PP45}
For any map $f :X\to \mathbb R$  and any box $B$ the canonical linear map $\mathbb F_r(B)\to \mathbb G_r(B)$ defined above is an isomorphism. 
$\mathbb F^f_r(B)= \mathbb G^f_r(B).$
\end{proposition}

\begin{proof}
Note that the injectivity is straightforward.
Indeed, suppose $\mathbb I_{a}(r)\cap \mathbb I^b(r) \ni x = x_1 +x_2$ with $x_1\in \mathbb I_{a'}(r)$ and  $x_2\in \mathbb I^{b'}(r)$. Then $x_1= x- x_2\in \mathbb \mathbb I^b(r)$ hence $x_1\in (\mathbb I_{a'}(r)\cap \mathbb I^b(r)) $ and similarly $x_2\in (\mathbb I_{a}(r)\cap \mathbb I^{b'}(r)).$

To check the surjectivity start with $x=x_1+ y_1= x_2+y_2$ s.t $x_1\in \mathbb I_{a'}, y_1\in \mathbb I^b$, $x_2\in \mathbb I_{a}, y_2\in \mathbb I^{b'}.$  Then $x-x_1-y_2$ is equivalent to $x$ in $\mathbb G_r(B).$  But $x-x_1- y_2= y_1-y_2= x_2- x_1$ hence it belongs to $\mathbb I^b$ and to $\mathbb I^a.$  

\end{proof}

Let $f:M^n\to \mathbb R$ be a  map, $M^n$ a $\kappa-$orientable closed topological manifold, and $a,b$ regular values such that the restriction of $f$ to $f^{-1}(a-\epsilon, a+\epsilon)$ and $f^{-1} (b-\epsilon, b+\epsilon)$ for a small enough positive $\epsilon$ are topological submersions. This makes $f^{-1}(a)$ and $f^{-1}(b)$ codimension one topological submanifolds of $M.$  

Let $i_a:M_a\to M,$ $i^b:M^b\to M,$ $j_a:M\to (M,M_a),$ $j^b:M\to (M, M^b)$  denote the obvious inclusions and $i_a(k), i^b(k), j_a(k), j^b(k)$ the inclusion induced linear maps for homology in degree $k$ and $r_a(k), r^b(k), s_a(k), s^b(k)$ the inclusion induced linear maps in cohomology, (with coefficients in the field $\kappa,$)  as indicated in the diagrams (\ref{PD1}) and (\ref{PD1'}) below. Poincar\'e  Duality provides the  commutative diagrams  (\ref{PD1}) and (\ref{PD1'}) with all vertical arrows isomorphisms.

\begin{equation}\label{PD1}
\xymatrix{&H_r( M_a)\ar[d]\ar[r]^{i_a(r)}  \quad &H_r( M)\ar[d]\ar[r]^{j_a(r)} &H_r( M, M_a)\ar[d]&\\
&H^{n-r}( M, M^a)\ar[d] \ar[r]^{s^a(n-r)} \quad&H^{n-r}( M)\ar[d] \ar[r]^{r^a(n-r))} &H^{n-r}( M^a)\ar[d]&\\
&(H_{n-r}( M, M^a))^\ast \ar[r]^{(j^a(n-r))^\ast} \quad&(H_{n-r}( M))^\ast \ar[r]^{(i^a(n-r))^\ast} &(H_{n-r}( M^a))^\ast&}
\end{equation}
\vskip .2in

\begin{equation}\label {PD1'}
\xymatrix
{&H_r( M^b)\ar[d] \ar[r]^{i^b(r)} \quad &H_r( M)\ar[d]\ar[r]^{j^{b}(r)} &H_r( M, M^b)\ar[d]&\\
&H^{n-r}( M, M_b)\ar[d] \ar[r]^{s_b(n-r)}  \quad &H^{n-r}( M)) \ar[d] \ar[r]^{r_b(n-r)} &H^{n-r}( M_b)\ar[d]\\
&(H_{n-r}( M, M_b))^\ast \ar[r]^{(j_b(n-r))^\ast} \quad &(H_{n-r}( M))^\ast \ \ar[r]^{(i_b(n-r))^\ast} &(H_{n-r}( M_b))^\ast &.}
\end{equation}
As a consequence of these two diagrams observe that Poincar\'e duality provide a canonical isomorphism 
\begin{equation}\label {EEE2}
\mathbb F^f_r(a,b)= (\mathbb G_{n-r}^f(b,a))^\ast.
\end{equation}
Indeed observe that: 
\begin{enumerate}
\item $\mathbb F_r(a,b)= \ker (j_a(r), j^b(r))$ by the exactness of the first rows in the diagrams (\ref{PD1}) and (\ref{PD1'}). 
Precisely $\ker (j_a(r), j^b(r))= \ker j_a(r) \cap j^b(r)= \mathbb I_a(r)\cap \mathbb I^b(r).$

\item $ \ker (j_a(r), j^b(r))\equiv \ker (r^a(n-r), r_b(n-r))$ by the isomorphism of the upper vertical arrows in these  diagrams.

\item $\ker (r^a(n-r), r_b(n-r))\equiv \ker ((i^a(n-r))^\ast, (i_b(n-r))^\ast))$ by the isomorphism of the lower vertical arrow in these diagrams. 

The isomorphisms above are induced by Poincar\'e duality and cohomology in terms of  homology ; Their composition is still referred to as 
Poincar\'e duality.

\item $\ker  ((i^a(n-r))^\ast, (i_b(n-r))^\ast)) = (\coker (i^a(n-r)+ i_b(n-r))^\ast = (\mathbb G^f_{n-r}(b,a))^\ast$ 
by standard finite dimensional linear algebra duality. 
\end{enumerate}
Putting together these equalities one obtains (\ref{EEE2}).
\vskip .1in

Suppose $M$ is a closed $\kappa-$orientable smooth manifold and $f:M\to \mathbb R$ a smooth map which is locally polynomial (i.e. in the neighborhood of any point, in some local coordinates is a polynomial). Such map is tame.   
For $(a,b)\in \mathbb R^2$ choose $\epsilon$  small enough so that the intervals $(a-\epsilon, a), (a, a+\epsilon)$ as well as $(a-\epsilon, a), (a, a+\epsilon)$ are contained in the set of regular values (in the sense of differential calculus). Such choice is possible in view of the tameness of $f.$

To establish the result as stated for such map we proceed as follows.
Observe that: 
\begin{enumerate}
\item In view of the tameness of $f$ 
\begin{equation}\label {EE1}
\hat \delta^f_r(a,b)= \mathbb F^f_r((a-\epsilon, a+\epsilon]\times [b-\epsilon, b+\epsilon)).
\end{equation}

\item By definition
\begin{equation}
\mathbb F^f_r((a-\epsilon, a+\epsilon]\times [b-\epsilon, b+\epsilon))= \coker \mathcal F_r((a-\epsilon, a+\epsilon]\times [b-\epsilon, b+\epsilon)).
\end{equation}
\item By proposition \ref {PP45}
\begin{equation} 
\coker \mathcal F_r((a-\epsilon, a+\epsilon]\times [b-\epsilon, b+\epsilon))= \ker (\mathcal G_r((a-\epsilon, a+\epsilon]\times [b-\epsilon, b+\epsilon)). 
\end{equation}
\item By equality (\ref{EEE1})
\begin{equation}
\ker (\mathcal G_r((a-\epsilon, a+\epsilon]\times [b-\epsilon, b+\epsilon))= (\coker  (\mathcal G_r((a-\epsilon, a+\epsilon]\times [b-\epsilon, b+\epsilon))^\ast))^\ast.
\end{equation}
\item  By equality (\ref{EEE2})
\begin{equation}
(\coker  (\mathcal G_r((a-\epsilon, a+\epsilon]\times [b-\epsilon, b+\epsilon))^\ast))^\ast= (\coker(\mathcal F_{n-r}((b-\epsilon, b+\epsilon]\times [a-\epsilon, a+\epsilon)))^\ast.
\end{equation}
\item In view of definition 
\begin{equation}
(\coker((\mathcal F_{n-r}((b-\epsilon, b+\epsilon]\times [a-\epsilon, a+\epsilon)))^\ast=(\mathbb F_{n-r}((b-\epsilon, b+\epsilon]\times [a-\epsilon, a+\epsilon)))^\ast.
\end{equation}
\item In view of tameness of $f$ 
\begin{equation}\label{EE2}
(\mathbb F^f_{n-r}((b-\epsilon, b+\epsilon]\times [a-\epsilon, a+\epsilon)))^\ast = (\hat\delta^f_{n-r}(b,a))^\ast.  
\end{equation}
\end{enumerate}

Putting together the equalities above one derive the result for $f$ as above. In view of Theorem \ref{T2} and the fact that locally polynomial maps are dense in the space of all continuous maps when $X$ is a smooth manifold the result holds as stated.

\vskip .2in

\subsection {A  comment }

 The hypothesis of compact ANR  can be replaced by ANR with total homology of finite dimension and proper map by {\it homologically proper map}  which means that for $I$ a closed interval the total homology of $f^{-1}(I)$ has finite dimension. All results remain unchanged  with essentially the same proof. 
An interesting situation when such generalization is relevant is the case of the absolute value  of the complex polynomial  function $f$  when restricted to the complement of its zeros,  which will be treated in future work, but can be easily reduced to the case of proper map considered above.

%
%
%
%

%
%
%


\end{document}